\documentclass{amsart}

\usepackage[foot]{amsaddr}
\usepackage[alphabetic]{amsrefs}
\usepackage[headings]{fullpage}
\usepackage{booktabs,subfigure,tikz}
\usepackage{algorithm, algorithmic}
\usepackage{color,hyperref,newtxtext,newtxmath}

\def\H{\mathcal H}
\def\R{\mathbb R}
\def\na{\nabla}
\def\x{\times}
\def\Om{\Omega}
\def\Id{\text{Id}}
\newcommand{\abs}[1]{\lvert#1\rvert}

\newcommand{\nm}[2]{\|#1\|_{#2}}
\def\dx{\,\mathrm{d}x}
\def\pa{\partial}
\newcommand{\mc}[1]{\mathcal{#1}}

\newcommand{\lr}[1]{\Bigl(#1\Bigr)}
\newcommand{\set}[2]{\left\{\,#1\,\mid\,#2\,\right\}}

\theoremstyle{plain}
\newtheorem{theorem}{Theorem}[section]

\newtheorem{example}[theorem]{Example}

\newtheorem{proposition}[theorem]{Proposition}

\begin{document}

\title[A pre-training DL for bilayer bending]{A pre-training deep learning method for simulating the large bending deformation of bilayer plates}%
\author[X. Li\and Y. L. Liao\and P. B. Ming]{Xiang Li\and Yulei Liao\and Pingbing Ming}
\email{lixiang615@lsec.cc.ac.cn, liaoyulei@lsec.cc.ac.cn, mpb@lsec.cc.ac.cn}
\address{LSEC, Institute of Computational Mathematics and Scientific/Engineering Computing, AMSS, Chinese Academy of Sciences, Beijing 100190, China}
\address{School of Mathematical Sciences, University of Chinese Academy of Sciences, Beijing 100049, China}
\thanks{Funding: This work was supported by National Natural Science Foundation of China [Grants No. 12371438 and 11971467].}

\begin{abstract}
We propose a deep learning based method for simulating the large bending deformation of bilayer plates. Inspired by the greedy algorithm, we propose a pre-training method on a series of nested domains, which accelerate the convergence of training and find the absolute minimizer more effectively. The proposed method exhibits the capability to converge to an absolute minimizer, overcoming the limitation of gradient flow methods getting trapped in the local minimizer basins. We showcase better performance with fewer numbers of degrees of freedom for the relative energy errors and relative $L^2$-errors of the minimizer through numerical experiments. Furthermore, our method successfully maintains the $L^2$-norm of the isometric constraint, leading to an improvement of accuracy.
\end{abstract}

\date{\today}
\maketitle


\section{Introduction}
In the field of solid mechanics, particularly in modern nanoscale processes, the bending and deformation of plates are classical problems. Linear models are commonly used to describe small deformations, while nonlinear models are employed to capture large deformations~\cite{Ciarlet:1997}. By letting the thickness of the plate tends to zero, a series of works~\cite{Muller:2002,muller2002,Friesecke:2006} have rigorously derived the dimensional reduced models from three-dimensional elasticity, under different scaling of energy with respect to the thickness, with the aid of $\Gamma$-convergence~\cite{Giorgi:1975,Dalmaso:1993}. Numerical methods for simulating the model formulated in~\cite{muller2002}, which coincides with the nonlinear Kirchhoff model~\cite{Kirchhoff1850}, have been proposed in~\cite{DKT13,bar13b,DG2021,LiMing:2023} along the same line. The main challenges lie in preserving isometry during deformation that keeping relations for length and angle unchanged, similar to a piece of paper.

Bilayer plates, which consist of two compound materials with slightly different properties, present greater complexity in terms of modeling and simulation, and have garnered significant attention. These structures find wide-ranging applications in aerospace, civil engineering, and materials science~\cite{Suzuki:1994,Smela:1995,Jager:2000,Schmidt:2001,Stoychev:2011,Kuo:2005,Bassik:2010,Jan:2016,Ye:2016}. When exposed to external heating or electrification, material mismatches between the two layers result in large bending deformations in bilayer plates. The study of bilayer plates involves the development of mathematical models and analytical techniques aimed at gaining a better understanding of their deformation. The classical model of bilayer~\cite{Timoshenko:1925analysis,Suhir:1986,Kuo:1989} can bend only in one direction with uniform curvature. Building upon the work~\cite{muller2002} that concerns curvature in two directions, certain mathematical models~\cite{schmidt2007minimal,Schmidt:2007} formulate the large bending deformation for multi-layer plates under different scaling of energy. Meanwhile a heuristic derivation for bilayer plates may be found in~\cite{Bar17bilayer}.

We call back the model for bilayer plates in~\cite{Bar17bilayer}: Given a domain $\Om\subset \R^{2}$ describing the middle surface of the reference configuration of the plate, and an intrinsic spontaneous curvature tensor $Z\in\R^{2\x 2}$,  which usually stands for the difference between the material properties of two layers, we minimize the elastic energy
\begin{equation}\label{eq:plate}
E[u]{:}=\frac{1}{2} \int_{\Om}\abs{H(u)+ Z}^{2}\dx-\int_{\Om} f \cdot u \dx
\end{equation}
within the isometric constraint for $u:\Om\to\R^3$, i.e.,
\begin{equation}\label{eq:isometry}
    [\na u(x)]^{\top} \na u(x)=\Id_{2},\quad\text{a.e.,\;} x\in\Om, 
\end{equation}
with the boundary $\Gamma_D\subset\partial \Om$ clamped, i.e.,
 \begin{equation}\label{eq:bc}
 u(x) = g,\quad \na u = \Phi\quad \text{on\;} \Gamma_D,
 \end{equation}
where $H(u)$ is the second fundamental form of the parametrized surface, i.e., $H_{ij}(u)=n\cdot\pa_i\pa_ju,n = \pa_1u \x \pa_2u,\Id_{2}$ is a two by two identity matrix, $f,g:\Omega\to\R^3$and $\Phi:\Omega\to\R^{3\times 2}$ satisfies $\Phi^\top\Phi=\Id_{2}$. 

This model, which can be viewed as a non-convex minimization problem with a nonlinear constraint, has been investigated numerically by \textsc{Bartels et al.}~\cite{Bar17bilayer,barThermal18,bartels2022stable} using the framework of $H^2$-gradient flow iteration. In each iteration, they discretize the sub-problem in the tangent space of the admissible set to linearize the nonlinear constraint. They employed various Kirchhoff triangles for spatial discretization, and established the $\Gamma$-convergence to the stationary configuration in theory. More recently, another work~\cite{bonito2020discontinuous} exploited the discontinuous Galerkin discretization to enhance the accuracy and flexibility of the algorithm. Moreover, a self-avoiding method, which ensures that the deformation does not contain self-intersections or self-contact, is proposed in~\cite{bartels22self-contact}. 

Unfortunately, due to the large deformation and the non-convexity of the energy, the equilibrium configurations obtained from the above methods are highly sensitive to the discretization mesh size $h$ and the pseudo-timestep $\tau$ in the gradient flow. The convergence may be guaranteed with sufficiently small $h$ and $\tau$, while naturally demands a high computing cost. In additional, the gradient flow algorithm may get trapped in certain local minimizer basins regardless of how parameters are chosen~\cite{Bar17bilayer}. For instance, when $Z=-5I_2$ and $\Omega=(-5, 5)\times(-2, 2)$ in~\eqref{eq:plate}, the method may never reach the absolute minimizer, and could be trapped in a local minimizer basin. This phenomenon can affect or prevent the bilayer plate from complete folding~\cite{Alben:2011}. 

In the present work, we propose a novel approach using deep neural network to solve the model for the bilayer plate~\eqref{eq:plate}, which may potentially outperform the traditional methods in tackling the non-convex optimization problem. Instead of imposing the isometric constraint at specific degrees of freedom in the finite element space, as done in previous methods~\cite{Bar17bilayer,barThermal18,bartels2022stable}, we incorporate the constraint by penalizing the energy functional through the addition of the $L^2$-norm of the isometric constraint~\eqref{eq:isometry}. Furthermore, we construct the neural network functions that precisely satisfy the boundary condition in~\eqref{eq:bc} to accurately handle the boundary conditions. In order to overcome the challenges posed by the large spontaneous curvature and the domain with high aspect ratio, we propose a pre-training method, which trains on a series of nested subdomains inspired by the greedy algorithm. This pre-training deep learning method is also effective in dealing with irregularly shaped domains. 

We present the performance of the proposed method for examples from~\cite{Bar17bilayer,bonito2020discontinuous} as well as relevant engineer literatures~\cite{Simpson:2010,Alben:2011,Jan:2016}. Our experiments demonstrate that our method exhibits superior performance compared to previous methods; See, e.g.,~\cite{Bar17bilayer,barThermal18,bonito2020discontinuous,bartels2022stable}. For the benchmark case that the absolute minimizer is corresponds to a  cylindrical shape, i.e., Example~\ref{ex:a1}, the relative error of the energy is approximately $10^{-3}$ and the relative $L^2$-error of the minimizer is about $10^{-2}$. In contrast, algorithms based on finite element discretization and $H^2$-gradient flow minimization only reduce the relative error of the energy to $10^{-1}$. Furthermore, we successfully maintain the $L^2$-norm of the isometric constraint at approximately $10^{-2}$. This accurate preservation of the isometric constraint significantly enhances the overall solution accuracy. Last but not least, one notable advantage of our method is its ability to converge to an absolute minimizer. On the other hand, the gradient flow method often gets trapped in local minimizer basins, especially in cases involving large spontaneous curvature and O-shaped domains. 

The remaining parts of this paper are organized as follows. In~\S~\ref{sec:penalty}, we impose the isometric constraint~\eqref{eq:isometry} by the penalty method and discuss on the origin energy~\eqref{eq:plate} and the equivalent energy introduced in~\cite{bonito2020discontinuous,bartels2022stable}. In \S~\ref{sec:DL} we present a detailed description of the deep learning approach applied to the model of bilayer plate~\eqref{eq:plate} and introduce a novel pre-training procedure based on domain decomposition. Next, in \S~\ref{sec:example}, we present a series of numerical experiments to demonstrate the superiority of our method in simulating the large bending deformation of bilayer plates. Finally, in \S~\ref{sec:conclusion}, we conclude the work with a short summary.
\section{Total energies with penalty}\label{sec:penalty}
To deal with the isometric constraint, for $\beta>0$, we use the penalty method and consider a new energy functional as
\begin{equation}\label{eq:energy}
    \min_{u\in H^2_{\Gamma_D}(\Omega)}I[u]{:}=E[u]+\beta C[u]^2,
\end{equation}
with the $L^2$-norm of the isometric tolerance $C[u]=\nm{\nabla u^\top\nabla u-\Id_2}{L^2(\Omega)}$, and 
\[H^2_{\Gamma_D}(\Omega)=\set{u\in H^2(\Omega)}{u=g,\quad \nabla u=\Phi\quad\text{on\,}\Gamma_D}.\]
In this work, we always assume $f=0$ since the source term $f$ is not the main trouble in the computation. Then the energy $E[u]$ is positive; i.e. $E[u]\ge 0$, and $E[u]\le \abs{\Omega}\abs{Z}^2$ if $g$ is linear and $\Phi$ is a constant matrix, where $\abs{\Om}$ is the area of $\Om$. To balance the penalty term, a rigorous choice of $\beta$ is $\mc{O}(\abs{\Omega}\abs{Z}^2/\delta^2)$ for the expected tolerance $C[u]<\delta$. In practice, as long as the tolerance $C[u]$ is small enough, we can use a penalty factor $\beta$ that is not as large as it should be, and it is sufficient to obtain an approximating solution.

Some works~\cite{bonito2020discontinuous,bartels2022stable} have used the property $\abs{H(u)}=\abs{D^2u}$ when the isometric constraint~\eqref{eq:isometry} holds and the energy~\eqref{eq:plate} is reshaped into
 \[
 \Tilde{E}[u]=\dfrac12\int_\Omega\abs{D^2u}^2\dx+\sum_{i,j=1}^2\int_\Omega\partial_{ij}u\cdot(\partial_1u\times\partial_2u)Z_{ij}\dx+\dfrac12\abs{\Omega}\abs{Z}^2.
 \]
Then the variational problem~\eqref{eq:energy} changes to
 \begin{equation}\label{eq:anotherenergy}
\min_{u\in H^2_{\Gamma_D}(\Omega)}\Tilde{I}[u] {:}= \Tilde{E}[u] + \beta C[u]^2,
\end{equation}
this formulation is not energy positive-preserving and leads to at least an $\mc{O}(1/\beta)$ energy error as in
 \begin{proposition}~\label{prop:prop1}
 There exists $\Tilde{u}\in H^2_{\Gamma_D}(\Omega)$ such that $\Tilde{I}[\Tilde{u}] <\Tilde{I}[u]$ and 
 \[
 \Tilde{E}[u]-\Tilde{E}[\Tilde{u}]\ge\dfrac{\abs{\Omega}}{5\beta},
 \]
 and
 \[
 \Tilde{I}[u]-\Tilde{I}[\Tilde{u}]\ge\dfrac{4\abs{\Om}}{625\beta}\qquad\beta\ge 1,
 \]
 where $u$ is the solution of~\eqref{eq:plate}, and $\Omega,Z,g,\Phi$ is bounded independent of $\beta$.
 \end{proposition}

 \begin{proof}
 For the solution $u$ that minimize the energy~\eqref{eq:plate} under isometric constraint~\eqref{eq:isometry}, it is straightforward that $\Tilde{I}[u]=I[u]\ge 0$. Consider the spontaneous curvature $Z$ with $Z_{11}=1$ and $Z_{ij}=0$ otherwise. Let
  \[\Tilde{u}=\left[\sin x_1,\lr{1+\dfrac1{5\beta}}x_2,\cos x_1\right]^\top,\]
  then a direct calculation gives
  \[\nabla\Tilde{u}=\begin{bmatrix}
      \cos x_1 & 0\\
      0 & 1+1/(5\beta)\\
      -\sin x_1 & 0 
  \end{bmatrix},\qquad \partial_{11}\Tilde{u}=\begin{bmatrix}
      -\sin x_1\\
      0\\
      -\cos x_1
  \end{bmatrix},\]
  and $\partial_{ij}\Tilde{u}=0$ otherwise. Therefore $\partial_1\Tilde{u}\times\partial_2\Tilde{u}=-[1+1/(5\beta)]\partial_{11}\Tilde{u}$ and
  \[\Tilde{E}[\Tilde{u}]=-\dfrac{\abs{\Omega}}{5\beta},\qquad C[\Tilde{u}]=\dfrac{10\beta+1}{25\beta^2}\sqrt{\abs{\Omega}}.\]
  Since 
  \[\Tilde{I}[\Tilde{u}]=\Tilde{E}[\Tilde{u}]+\beta C[\Tilde{u}]^2\le-\dfrac{4\abs{\Omega}}{625\beta},\] when $\beta\ge 1$,  we conclude that the lower bound of the energy error is at least $\mc{O}(1/\beta)$.
 \end{proof}

In practice, we always choose $\beta$ between $100-1000$. Otherwise the gradient descent algorithms will converge too slowly if $\beta$ is too large. The following proposition shows that the minimizer of the problem~\eqref{eq:anotherenergy} may be completely wrong, even if the tolerance error is small enough.
\begin{proposition}~\label{prop:prop2}
For fixed penalty factor $\beta$ and some $\gamma>1/4$, there exists $\Tilde{u}\in H^2_{\Gamma_D}(\Omega)$ such that $\Tilde{I}[\Tilde{u}] <\Tilde{I}[u]$ and the tolerance $C[\Tilde{u}]$ is $\mc{O}(\beta^{-2\gamma})$, while 
\[
\Tilde{E}[u]-\Tilde{E}[\Tilde{u}]\ge\abs{\Omega},
\]
and
\[
\Tilde{I}[u]-\Tilde{I}[\Tilde{u}]>(1-9\beta^{1-4\gamma})\abs{\Omega}\ge0
\]
provided that $\beta\ge 9^{1/(4\gamma-1)}$, where $u$ is the solution of~\eqref{eq:plate}, and $\Omega,g,\Phi$ is bounded independent of $\beta$ while $\abs{Z}=\beta^\gamma$.
\end{proposition}

\begin{proof}
Consider the spontaneous curvature $Z$ with $Z_{11}=\beta^\gamma$ and $Z_{ij}=0$ otherwise. Let
  \[\Tilde{u}=\left[\beta^{-\gamma}\sin(\beta^\gamma x_1),(1+\beta^{-2\gamma})x_2,\beta^{-\gamma}\cos(\beta^\gamma x_1)\right]^\top.\]
Proceeding along the same line as Proposition~\ref{prop:prop1}, a direct calculation gives
  \[\nabla\Tilde{u}=\begin{bmatrix}
      \cos(\beta^\gamma x_1) & 0\\
      0 & 1+\beta^{-2\gamma}\\
      -\sin(\beta^\gamma x_1) & 0 
  \end{bmatrix},\qquad \partial_{11}\Tilde{u}=-\beta^\gamma \begin{bmatrix}
      \sin(\beta^\gamma x_1)\\
      0\\
      \cos(\beta^\gamma x_1)
  \end{bmatrix},\]
  and $\partial_{ij}\Tilde{u}=0$ otherwise. Therefore $\partial_1\Tilde{u}\times\partial_2\Tilde{u}=-\beta^{-\gamma}(1+\beta^{-2\gamma})\partial_{11}\Tilde{u}$ and
  we obtain
\[\Tilde{E}[\Tilde{u}]=-\abs{\Omega},\qquad C[\Tilde{u}]=\beta^{-2\gamma}(2+\beta^{-2\gamma})\sqrt{\abs{\Omega}}.\]
Since 
\[\Tilde{I}[\Tilde{u}]=\Tilde{E}[\Tilde{u}]+\beta C[\Tilde{u}]^2<(9\beta^{1-4\gamma}-1)\abs{\Omega}\] and is $\mc{O}(1)$ for large $\beta$,  we conclude that the lower bound of the energy error is at least $\mc{O}(1)$ if $\abs{Z}$ is $\mc{O}(\beta^\gamma)$.
\end{proof}

Based on the above discussion, we shall use the total energy~\eqref{eq:energy} in the following work, which is much stable than~\eqref{eq:anotherenergy} when the spontaneous curvature $Z$ is large.
\section{Deep Learning with pre-training}\label{sec:DL}
In the framework of deep learning, we use Stochastic Gradient Descent (SGD) to minimize the total energy~\eqref{eq:energy} over the trial set $\H\subset H^2_{\Gamma_D}(\Omega)$ modeled by ResNet~\cite{he2016deep}; i.e., $\min I[\Hat{u}]$ for $\Hat{u}\in\H$, here $\H$ the set with neural network functions will be declared specifically later. The ResNet we use is plot in Figure~\ref{Fig.resnet}, which consists of a fully connected input layer, $l$ residual blocks and a fully connected output layer. Each residual block has two fully connected layers with $m$ hidden nodes. 
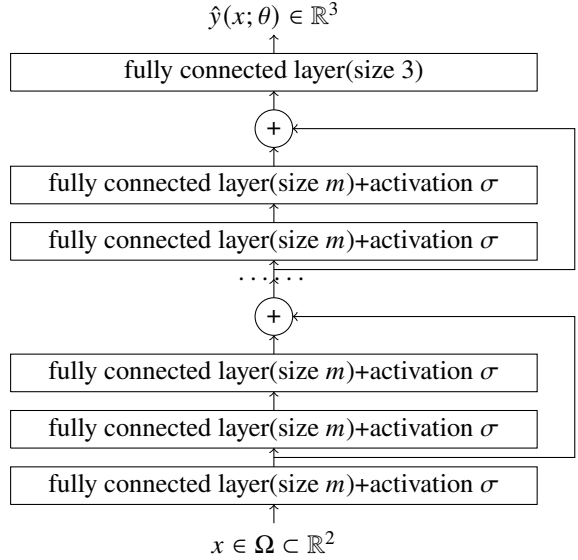
\begin{figure}[htbp]\centering\begin{tikzpicture}
  \node at(0,0) {$x\in\Omega\subset\R^2$};
  \draw[->](0,0.25)--(0,0.5);
  \draw(-3.5,0.5)--(3.5,0.5)--(3.5,1)--(-3.5,1)--(-3.5,0.5);
  \node at(0,0.75){fully connected layer(size $m$)+activation $\sigma$};
  \draw[->](0,1)--(0,1.25);
  \draw(-3.5,1.25)--(3.5,1.25)--(3.5,1.75)--(-3.5,1.75)--(-3.5,1.25);
  \node at(0,1.5){fully connected layer(size $m$)+activation $\sigma$};
  \draw[->](0,1.75)--(0,2);
  \draw(-3.5,2)--(3.5,2)--(3.5,2.5)--(-3.5,2.5)--(-3.5,2);
  \node at(0,2.25){fully connected layer(size $m$)+activation $\sigma$};
  \draw[->](0,2.5)--(0,2.75);
  \draw(0,3)circle[radius=0.25];
  \node at(0,3){+};
  \draw[->](0,1.125)--(4,1.125)--(4,3)--(0.25,3);
  \draw[->](0,3.25)--(0,3.5);
  \node at(0,3.5){……};
  \draw[->](0,3.5)--(0,3.75);
  \draw(-3.5,3.75)--(3.5,3.75)--(3.5,4.25)--(-3.5,4.25)--(-3.5,3.75);
  \node at(0,4){fully connected layer(size $m$)+activation $\sigma$};
  \draw[->](0,4.25)--(0,4.5);
  \draw(-3.5,4.5)--(3.5,4.5)--(3.5,5)--(-3.5,5)--(-3.5,4.5);
  \node at(0,4.75){fully connected layer(size $m$)+activation $\sigma$};
  \draw[->](0,5)--(0,5.25);
  \draw(0,5.5)circle[radius=0.25];
  \node at(0,5.5){+};
  \draw[->](0,3.625)--(4,3.625)--(4,5.5)--(0.25,5.5);
  \draw[->](0,5.75)--(0,6);
  \draw(-3.5,6)--(3.5,6)--(3.5,6.5)--(-3.5,6.5)--(-3.5,6);
  \node at(0,6.25){fully connected layer(size $3$)};
  \draw[->](0,6.5)--(0,6.75);
  \node at(0,7){$\Hat{y}(x;\theta)\in\R^3$};
\end{tikzpicture}
\caption{The component of ResNet.}\label{Fig.resnet}
\end{figure}
\subsection{Boundary condition}
There are several ways to impose the boundary conditions in the variational problem, such as the penalized boundary condition in~\cite{EYu:2018,Sirignano:2018}, the one based on the Nitsche variational principle~\cite{LiaoMing:2021} and the {\em exact} boundary condition~\cite{berg2018unified,Yang:2021}, among many other things. We follow~\cite{berg2018unified} and construct two functions 
$g_1(x): \Om\to\R$ and $g_2(x):\Om\to\R^3$  
\begin{equation}\label{eq:g1}\begin{aligned}
      g_1|_{\Gamma_D} = 0, \quad g_1|_{\Omega\backslash\Gamma_D} \neq 0&\quad \na g_1|_{\Gamma_D} = 0,\\
      \quad g_2|_{\Gamma_D} = g, &\quad \na g_2|_{\Gamma_D} = \Phi,
\end{aligned}\end{equation}
satisfying the clamped boundary conditions. The trial set $\Hat{u}\in\H$ is defined as
\[
\hat{u}(x; \theta){:}= g_1(x)\hat{y}(x;\theta) + g_2(x)
\]
such that
\[
\hat{u} = g, \quad\na_x \hat{u} = \Phi\quad \text{on} \ \Gamma_D,
\]
where $\hat{y}(x;\theta):\Om\to\R^3$ is the output of ResNet in Figure~\ref{Fig.resnet}. For example, when the plate is clamped along the edge $x_1=0$, we use $g_1=x_1^2$ and
\[
  \hat{u}(x; \theta)  = x_1^2\hat{y}(x;\theta) + \begin{bmatrix}
    x_1\\
    x_2\\
    0
  \end{bmatrix},
\]
to satisfy the clamped boundary condition with
\[
  g= 
  \begin{bmatrix}
    0\\
    x_2\\
    0
  \end{bmatrix}, \qquad
  \Phi =
  \begin{bmatrix}
    1&0\\
    0&1\\
    0&0
  \end{bmatrix}.
\]

The boundary functions may not be so easily to construct in some cases, we may employ the neural networks to approximate $g_1$ and $g_2$ as~\cite{berg2018unified}, in which $g_1$ is constructed to approximate the distance function to the clamped boundary. This may cause difficulties since the distance function is not smooth enough, thus~\cite{Yang:2021} employs the radial basis interpolation on a group of points on some segments of the boundary. We combine the advantages of both approaches by training $g_1$ that satisfies
\[
     g_1|_{\Gamma_D} = 0, \quad \na g_1|_{\Gamma_D} = 0, \quad g_1|_{\partial\Omega\backslash\Gamma_D} = d(x),
\]
where $d(x)>0$ and $\abs{\partial_t d(x)}\sim\mc{O}(1)$, which mimics the distance function. The effect of this approach will be demonstrated in Example~\ref{ex:oc}.
\subsection{Loss function}
In each epoch, we randomly sample $N_i$ points $\{x_k\}$ in $\Omega$, and use Monte Carlo method to approximate the integration, i.e.,
\[
\int_{\Om} f(x) \dx=\lim _{N \to\infty} \dfrac{\abs{\Om}}{N} \sum_{k=1}^N f(x_k).
\]
We define the loss function $I^*_{\Omega}[\hat{u}]$ as
\begin{equation}\label{loss function}
\begin{split}
I^*_{\Omega}[\hat{u}]:=\frac{\abs{\Om}}{N_i} \sum_{k=1}^{N_i}&\bigg(\frac{1}{2}
\abs{H(\hat{u}(x_k;\theta))+Z }^2 - f(x_k)\hat{u}(x_k;\theta)\\
&\quad +\beta\abs{\na_x\hat{u}(x_k;\theta)^\top\na_x\hat{u}(x_k;\theta) -\Id_2}^2\bigg),
\end{split}
\end{equation} 
where is the batch size of $i$-th iteration.
\subsection{Pre-training}
Numerical experiments conducted in previous studies such as~\cite{Bar17bilayer,barThermal18,bonito2020discontinuous,bartels2022stable} highlighted the challenges associated with solving the bilayer model under certain isometry constraint. The factors such as a large spontaneous curvature $Z$, the domain $\Om$ with high aspect ratio, and the non-convex O-shape domain significantly increase the difficulty. These phenomena indicate that the degree of non-convexity of the energy functional $E[u]$ is closely related to properties of the domain $\Omega$. Motivated by these observations, we propose an algorithmic improvement by dividing the domain $\Omega$ into a series of subdomains $\{\Omega_i\}_{i=1}^n$: $\Gamma_D \subset \Omega_1 \subset \dots \subset\Omega_n = \Omega$. The idea is to expect that minimizing $I^*_{\Omega_i}[\Hat{u}]$ on the subdomain $\Omega_i$ becomes easier than minimizing $I^*_{\Omega_{i+1}}[\Hat{u}]$ on the subdomain $\Omega_{i+1}$. We will present numerical experiments in the next section to validate the effectiveness of this pre-training method on reducing the number of iterations.

Denote by $D=\text{conv}\Omega$ the convex hull of $\Omega$, we divide $D$ into a series of convex subdomains $\{D_i\}_{i=1}^n:\Gamma_D\subset D_1\subset\dots\subset D_n=D$, and then select each subdomain as $\Omega_i=D_i\cap\Omega$. The objective is to train a neural network that minimizes the functional~\eqref{loss function} successively, starting with $\Omega_1$, and subsequently using the well-tuned parameters of the neural network as the initial values for training procedure on $\Omega_2$, and so forth. This process continues until the initial parameters of the neural network for training on the entire domain $\Omega$ are obtained. Algorithm~\ref{pre-training} provides a detailed description of this approach.
\begin{algorithm}[htpb]
    \caption{Deep learning with pre-training}
    \label{pre-training}
    \begin{algorithmic}[1]
      \REQUIRE The number of subdomains $n$, the penalty parameter $\beta$, max pre-training iteration number $Epoch_{pre}$ on each subdomain,  max training iteration number $Epoch$ on domain $\Omega$.  
            \STATE $D=\text{conv}\Omega$.
            \STATE Divide $D$ into a series of convex subdomains $\{D_i\}_{i=1}^n:\Gamma_D\subset D_1\subset\dots\subset D_n=D$.  
            \STATE $\Omega_i=D_i\cap\Omega$ for $i=1,\dots,n$. 
      \FOR{$i = 1, \dots, n - 1$}
      \STATE  Train the network $Epoch_{pre}$ times on subdomain $\Omega_i$ with loss function $  I^*_{\Omega_i}[\hat{u}]$~ \eqref{loss function}.
      \ENDFOR
      \STATE  Train the network $Epoch$ times on domain $\Omega$ with loss function $ I^*_{\Omega}[\hat{u}]$~\eqref{loss function}.
    \end{algorithmic}
\end{algorithm}%

This algorithm may be understood to solve problems from the perspective of information flow within subdomains. In this algorithm, the values of $\Hat{u}$ within the subdomain $\Omega_i$ are treated as information. Initially, the problem only has boundary conditions on $\Gamma_D$ as the available information. Therefore we start by seeking information on  $\Omega_1$, which contains $\Gamma_D$. Once we have the information on $\Omega_i$, we apply the concept of the greedy algorithm to expand the information to $\Omega_{i+1}$. By applying this process for each subdomain, we gradually gather more and more information until we have covered the entire domain $\Omega$.
\section{Numerical examples}\label{sec:example}
We use ResNet with five residual blocks and each block consists of two fully connected layers with ten hidden nodes; See Figure~\ref{Fig.resnet}, which has around $1050$ parameters. In our numerical experiments, such a small network is easy to train and the accuracy is satisfactory. It seems unnecessary to construct a larger network. 

We choose a smooth activation function \(\sigma(x) = \text{tanh}x\). The model is trained by the Adam optimizer~\cite{adam2014} with a fixed learning rate $0.001$. We set the maximum number of iterations $Epoch$ as $1e6$, and set the batch size as $16*\abs{\Om}$ unless otherwise stated, which may take a few hours to train.
  
We decompose the loss function $I^{\ast}[\hat{u}]$ into two parts, energy $E[\hat{u}]$ and $L^2$-norm of the isometry tolerance $C[\hat{u}]$, 
\[
I^{\ast}[\hat{u}]  = E[\hat{u}] + \beta C[\hat{u}]^2.
\]
During the test process, we uniformly sample $1e4*|\Omega|$ points to approximate $E[\hat{u}]$ and $C[\hat{u}]$. Moreover, we report the relative $L^2$-error 
\[
    e_{L^2} = \dfrac{\|u - \hat{u}\|_{L^2(\Omega)}}{\|u\|_{L^2(\Omega)}},
\]
 whenever the global minimizer $u$ is available.
 \subsection{Clamped plate: isotropic curvature}
 We begin with the example with an isotropic spontaneous curvature $Z=-\alpha\,\Id_2$ whose absolute minimizer is proved to be a cylinder of radius $1/\alpha$ in~\cite{schmidt2007minimal}. The plate $\Om=(-5,5)\x (-2,2)$ is clamped on the left side $\Gamma_D=\{-5\} \times[-2,2]$ with the boundary condition
\[
u= \begin{bmatrix}
      -5\\
            x_2\\
            0
  \end{bmatrix}_{3\times 1}, \quad \nabla u=
\begin{bmatrix}
    \textup{Id}_{2}\\
    0
\end{bmatrix}_{3\times 2} \quad \text { on } \Gamma_D.
\]
It follows from~\cite{schmidt2007minimal} that the problem has an absolute minimizer
\begin{equation}\label{exact sol}
u(x_1, x_2) = \left[\dfrac{1}{\alpha}\sin(\alpha(x_1 + 5)) - 5,\quad x_2,\quad \frac{1}{\alpha}\lr{1 - \cos(\alpha(x_1 + 5))}\right]^{\top},
\end{equation}
with energy of 
\begin{equation}
    E[u]=20\alpha^2.
\end{equation} 
We impose the boundary condition on $\hat{u}(x;\theta)$ by  
\[
\hat{u}(x; \theta) = (x_1 + 5)^2 \hat{y}(x;\theta) + \begin{bmatrix}
      x_1\\
      x_2\\
      0
\end{bmatrix}. 
\]

We test $\alpha=1, 2.5, 5$ and $10$. A larger $\alpha$ implies a smaller radius of the cylinder, resulting in increased difficulty in bending the plate. In~\cite{Bar17bilayer}, the author employed the DKT element and the $H^2$-gradient flow approach to solve this problem for $\alpha=1$ and $\alpha=5$. Their methodology achieves a cylindrical shape with an energy of $18.062$ when $\alpha= 1$, although the theoretical optimal energy is $20$. For $\alpha = 5$, the authors in~\cite{Bar17bilayer} discovered only an equilibrium configuration known as the {\em dog ears}. In a subsequent work~\cite{bonito2020discontinuous}, the authors exploited dG discretization and an $H^2$-gradient flow to simulate the same problem. They obtained a cylinder with energy $18.891$ when $\alpha = 1$, which slightly surpasses the results in~\cite{Bar17bilayer}. Furthermore, the authors in~\cite{bartels2022stable} simulated the same problem for $\alpha = 2.5$, resulting in a cylindrical shape with energy $80.461$, whereas the absolute minimum energy is $125$ for $\alpha=2.5$. 
 \begin{example}\label{ex:a1}
 $Z=-\textup{Id}_2$.
 \end{example}
This serves as the most rudimentary case for a bilayer plate, as delineated earlier. The absolute minimizer is a cylinder of radius $1$ with energy $20$. The authors in~\cite{Bar17bilayer} obtained this configuration asymptotically by DKT-based discretization and an $H^2$-gradient flow algorithm. By exploiting the dG discretization, the authors in~\cite{bonito2020discontinuous} also found the absolute minimizer with higher energy accuracy compared to DKT approach in~\cite{Bar17bilayer}. We list the results in these papers in Table~\ref{tab.compareDGDKT} for the sake of comparison.
\begin{table}[http]
    \centering
    \begin{tabular}{ccccccccccccc}
      \toprule
      \text { Mesh } & &\#4 & && & \# 5 & && &\#6 & \\
      \midrule
      \text { Method } & dG & & DKT& & dG &  &DKT& & dG & & DKT \\
          \text{DoFs} & 7680 & & 2601& & 30720 &  &9801& & 122880 & & 38025 \\
      \text { Energy } &18.514&&15.961& &18.679 & & 16.544& & 18.891 & &No Convergence\\
      \bottomrule
\end{tabular}
\caption{The final energy and degree of freedoms of DKT~\cite{Bar17bilayer} and dG~\cite{bonito2020discontinuous} on the meshes with different level of uniform refinements.}\label{tab.compareDGDKT}
\end{table}

We display the pseudo-evolution of the bilayer towards the absolute minimizer in Figure~\ref{fig:a1deformation}, which shows that the plate rolls into a cylinder of radius $1$ with energy $20.15$ after $2e5$ steps of training. This confirms the results in~\cite{schmidt2007minimal,Bar17bilayer,bonito2020discontinuous}. 
\begin{figure}[htbp]
\centering
\includegraphics[width=0.9\linewidth]{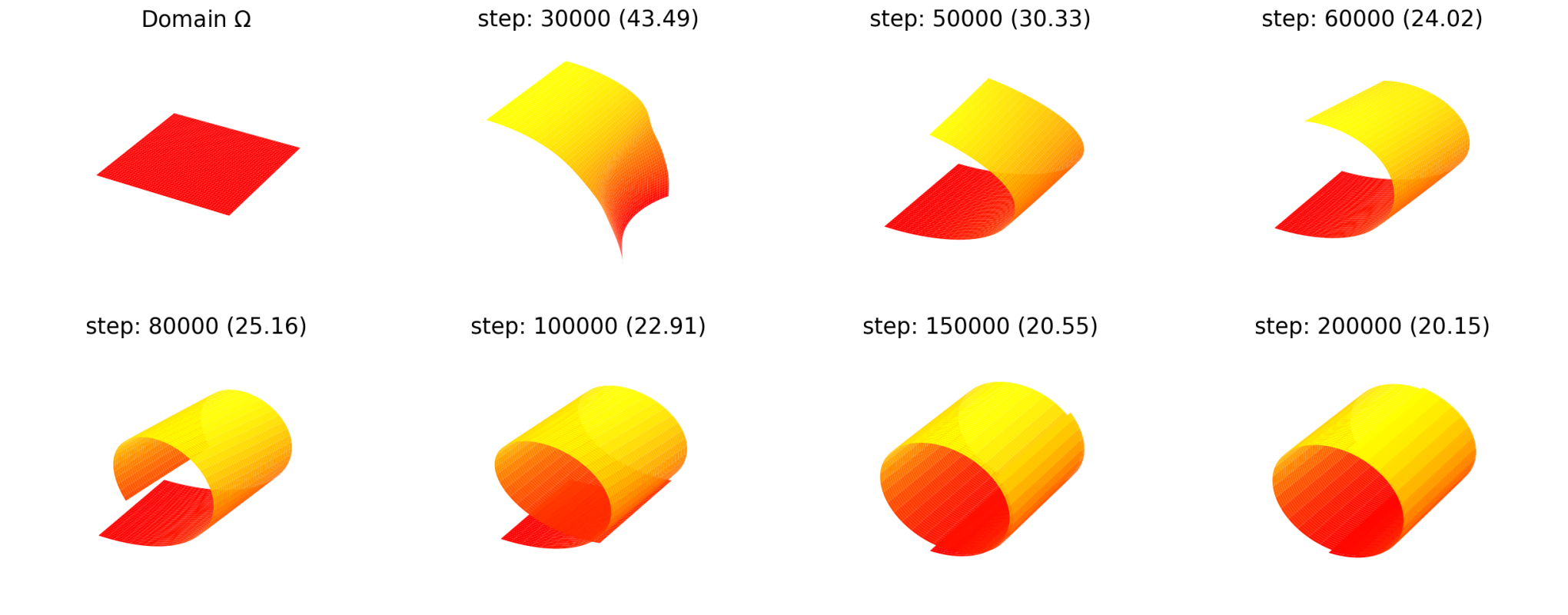}
 \caption{$\alpha = 1$ and $\beta = 500$. The pseudo-evolution of the bilayer plate reaches the absolute minimizer with energy $20$. The number of the iteration steps and the corresponding energy are reported.}\label{fig:a1deformation}
\end{figure}
\begin{table}[http]
    \centering
  \begin{tabular}{ccccc}
        \toprule
        $\beta$ & $E[\hat{u}]$ & $C[\hat{u}]$ & $e_{L^2}$ &Cylindrical shape\\
        \midrule
        100 & 17.21 & 1.14e-1 & 7.08e-2& Y\\
        500  & 20.03 & 2.58e-2& 4.42e-2&Y\\
        1000   & 20.15&  2.36e-2&5.54e-2& Y\\
        \bottomrule
    \end{tabular}
\caption{Test error, $\alpha = 1$, we directly train the network for $2e5$ steps with different $\beta$, and report the energy, $L^2$-isometric error, $L^2$-relative error and whether the plate reaches a cylindrical shape. The total number of the parameters of the neural network is $1050$.}\label{tab.eg1dl}
\end{table}
In view of Table~\ref{tab.compareDGDKT} and Table~\ref{tab.eg1dl}, our method admits an relative energy error of $1.50e-3$, which is an order of magnitude improvement compared to the previous best result of $5.55e-2$ reported in~\cite{bonito2020discontinuous}.

In our approach, we need to select the appropriate penalty factor $\beta$. The larger the value of the penalty factor $\beta$, the higher the accuracy, but the number of training steps increases accordingly. $\beta=500$ or $\beta=1000$ seems a good choice. We shall leave more discussion on the selection of $\beta$ for more complicated examples later on, and we will show that when the pre-training algorithm is applied, the penalty factor $\beta$ may be chosen relatively small while the algorithm still achieves high accuracy.
\begin{example}
$Z=-2.5\textup{Id}_2$.
\end{example}

\textsc{Schmidt}~\cite{schmidt2007minimal} has shown that the global minimizer of~\eqref{eq:plate} for this problem is given by a cylinder of height $4$ and radius $0.4$ with an energy $125$. We test the method with penalty factors $\beta$ varying from $10$ to $5000$, and train the network directly without pre-training. The pseudo-evolution configurations of the deformed plate are plot in Figure~\ref{fig:a2p5deformation}. It is clear that the cylindrical shape is reached after $3e5$ steps iteration. The resulting plate energy is $125.03$, which is very close to the theoretical minimal energy $125$, and is much better than the energy $107.05$ reported in~\cite{bartels2022stable}.
\begin{figure}[htbp]
\centering
\includegraphics[width=0.9\linewidth]{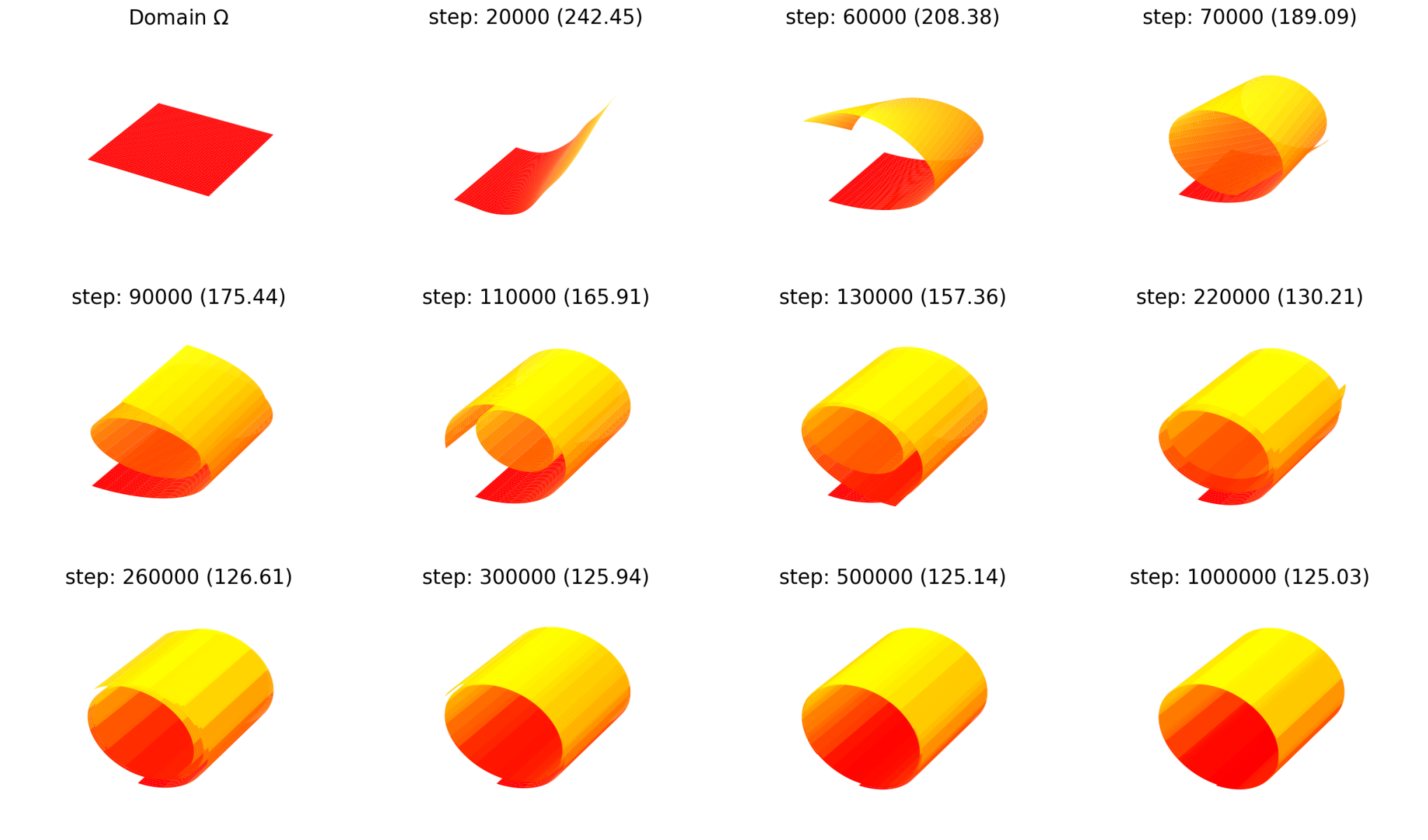}
\caption{$\alpha = 2.5$ and $\beta = 1000$. We train the network directly and plot the evolution of the bilayer plate, and report
the number of the iteration step and the corresponding energy. After about $3e5$ SGD iteration steps, the plate reaches a cylindrical shape.}\label{fig:a2p5deformation}
\end{figure}
  
The test errors are shown in Table~\ref{tab.finalA2p5}. We have reached the cylindrical shape for all $\beta$. A small penalty factor $\beta$ gives a small energy and a large isometric error, while the relative $L^2$-error seems not reasonably accurate. When $\beta=500$, we get  the final energy $125.01$, $L^2$-isometric error $2.44e-2$, and the relative error $4.81e-3$, which are much more accurate than the same results reported in~\cite{bartels2022stable}.
%
\begin{table}[http]
    \centering
    \begin{tabular}{ccccc}
        \toprule
        $\beta$ & $E[\hat{u}]$ & $C[\hat{u}]$ & $e_{L^2}$ &Cylindrical shape\\
        \midrule
        10 & 102.05& 6.94e-1 & 6.39e-2& Y\\
        100 & 112.36& 1.80e-1 & 4.03e-2& Y\\
        500  & 125.01 & 2.44e-2& 4.81e-3&Y\\
        1000   & 125.03&  1.76e-2& 1.55e-2& Y\\
        5000   & 125.69 &1.61e-2& 7.82e-2 &Y\\
        \bottomrule
    \end{tabular}
\caption{Test error, $\alpha= 2.5$. We report the energy, $L^2$-isometric error, $L^2$-relative error and whether the plate reaches a cylindrical shape after $1e6$ iteration.}\label{tab.finalA2p5}
\end{table}
\begin{example}
$Z = -5\textup{Id}_2$.
\end{example}

\textsc{Schmidt}~\cite{schmidt2007minimal} has shown that the global minimizer of~\eqref{eq:plate} for this problem is given by a cylinder of height $4$ and radius $0.2$ with an energy of $500$. This example has been test in~\cite{Bar17bilayer}, in which an equilibrium configuration likes {\em dog ears} has been presented.

We illustrate the pseudo-evolution of the plate in Figure~\ref{fig:a5deformation}. The plate rolls to a cylindrical shape after $1.5e6$ iteration. It is noticeble that the equilibrium configurations similar to {\em dog ears} appear as the intermediate states when the iteration steps vary from $1.8e5$ to $3.0e5$. In view of Figure~\ref{fig:a5deformation}, Figure~\ref{fig.traina5} and Figure~\ref{fig. testa5}, the energy drops very slowly when the number of the iteration step is larger than $5e5$, while the deformation and the relative $L^2$-error change dramatically. As reported in~\cite{Bar17bilayer}, the energy corresponding to the {\em dog ears} is $518.897$, which is not far from the theoretical minimal energy $500$, while the configuration is quite different from the final cylindrical shape.
\begin{figure}[htbp]
  \centering
  \includegraphics[width=0.9\linewidth]{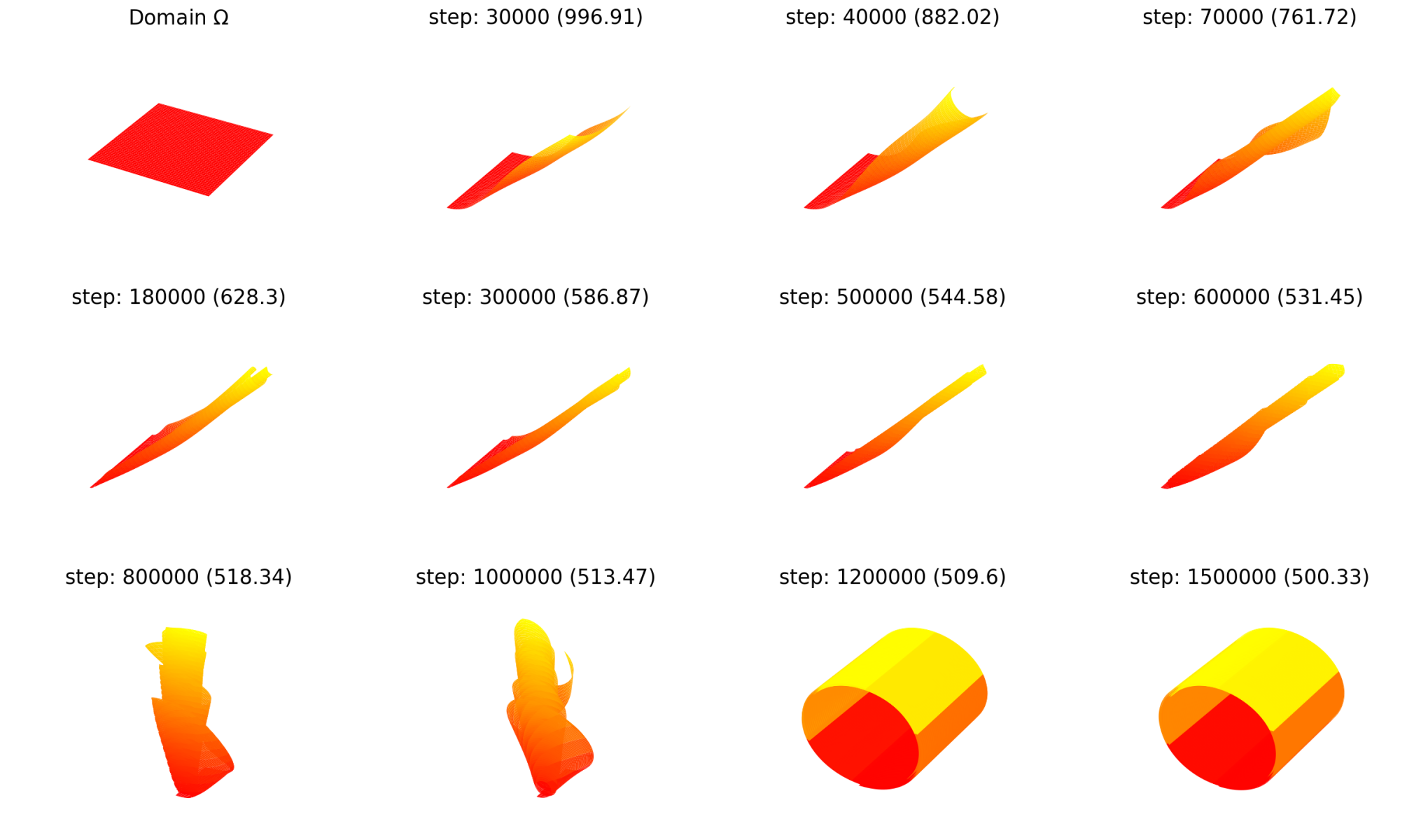}
\caption{$\alpha = 5$, $\beta = 1000$. We train the network directly and plot the pseudo-evolution of the bilayer plate. The iteration steps and energy corresponding to the deformation are shown above. After about $1,200,000$ SGD iteration, the plate reaches a cylindrical shape, which is the absolute minimizer. The energy is $500$.}\label{fig:a5deformation}
\end{figure}
\begin{figure}[htbp]
  \centering
  \subfigure[]{
    \includegraphics[width=0.45\linewidth]{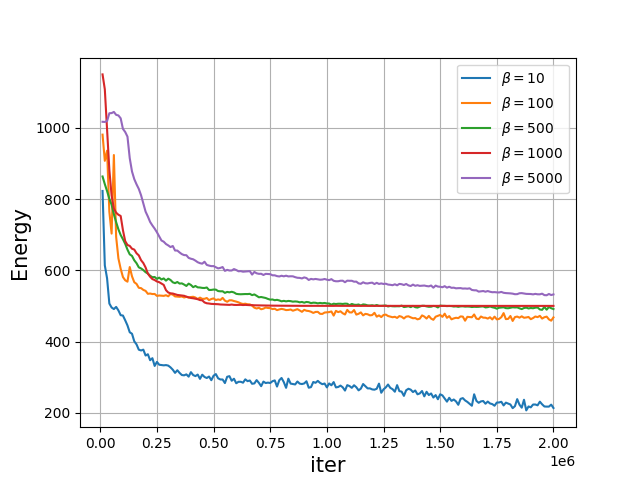}
    \label{Figa5E}
  }
  \subfigure[]{
    \includegraphics[width=0.45\linewidth]{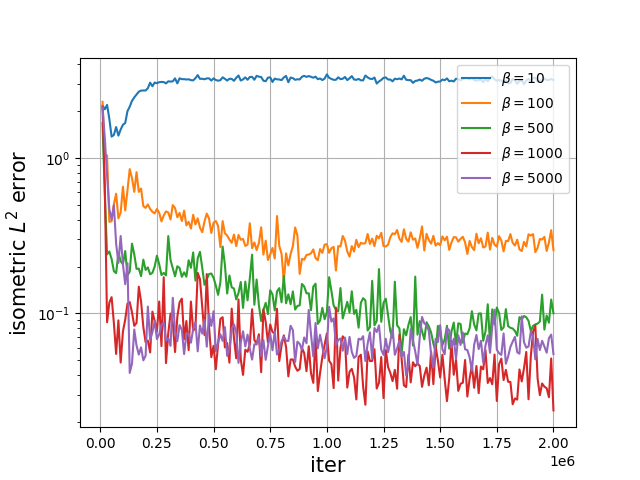}
    \label{Figa5I}
  }
  \caption{Training loss, $\alpha= 5$. Decay of energy $E[\hat{u}]$ and $L^2$-isometric error $C[\hat{u}]$ for different penalized parameter $\beta$ during the training process.}\label{fig.traina5}
\end{figure}
\begin{figure}[htbp]
\centering
\includegraphics[width=0.6\linewidth]{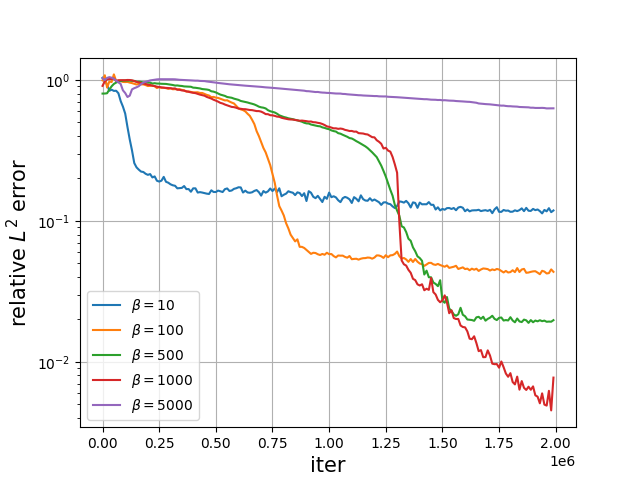}
\caption{Test error, $\alpha= 5$. For every $1e4$ steps of iteration, we report the $L^2$-relative error. When $\beta$ is properly chosen, compared to $\alpha=2.5$, much more steps are required to reduce $L^2$-error below $5e-2$ for $\alpha=5$, which motivates us to pre-train the network to find a good initial deformation to accelerate convergence.}\label{fig. testa5}
\end{figure}

We report the approximating energy, the isometrical discrepancy error and the relative $L^2$-error in Table~\ref{tab.finalA5}. It seems the method with very small $\beta$ fails to give the desired cylindral shape, while method with very large $\beta$ is also not so accurate. The method with $\beta=500$ and $\beta=1000$ leads to satisfactory results, the convergence behavior is similar to that of $\alpha=2.5$, while the number of the iteration doubles. We shall use Pre-training method (Algorithm \ref{pre-training}) for speed-up.
\begin{table}[htbp]
\centering
\begin{tabular}{ccccc}
\toprule
    $\beta$ & $E[\hat{u}]$ & $C[\hat{u}]$ & $e_{L^2}$ &Cylindrical shape\\
    \midrule
    10      &222.42 & 3.11 &1.19e-1&N \\
    100    & 465.96& 2.93e-1 &4.35e-2&Y \\
    500    & 500.04 & 8.18e-2 & 1.97e-2&Y \\
    1000  & 500.11& 5.39e-2 & 7.70e-3 &Y\\
    5000  & 531.19& 5.21e-2 & 6.29e-1 &Y\\
    \bottomrule
\end{tabular}
\caption{Test error, $\alpha = 5$. We report the energy, $L^2$-isometric error, $L^2$-relative error and whether the plate reaches a cylindrical shape after $2e6$ iteration steps. As $\alpha$ increases from $2.5$ to $5$, we obtain the similar $L^2$-relative error with exact solution, but more iteration steps are required.}\label{tab.finalA5}
\end{table}

We split the domain equally into $5$ blocks; see Figure~\ref{fig:DDrec}, take $Epoch_{pre}=50000$ and apply Algorithm \ref{pre-training}. In view of Figure~\ref{fig. pretraina5}, the pre-training provides a good initial guess for the iteration. In view of Table~\ref{tab.prefinala5} and Figure~\ref{fig:prea5}, the isometric constraint errors are lower for small $\beta$ such as $100$ and $500$ when the number of the pre-training step is larger than $2e5$, and the relative $L^2$-error may be improved by an order of magnitude. We achieve better accuracy with only a quarter of the number of the iteration compared to the direct training.
\begin{figure}[htbp]
  \centering
  \includegraphics[width=0.6\linewidth]{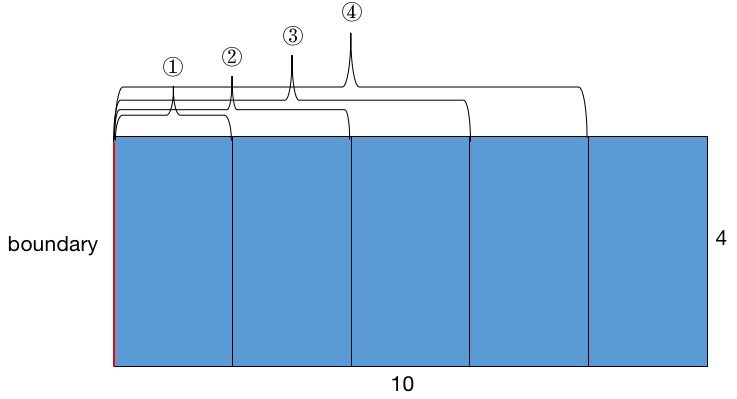}
\caption{To pre-train the network, we divide the domain $(-5,5) \times   (-2,2) $ equally into $5$ parts. We first randomly sample points in subdomain $1$ and train $Epoch_{pre}$ steps, and then continue the same process in subdomains $2, 3$, and $4$.}
  \label{fig:DDrec}
\end{figure}
\begin{figure}[htbp]
  \centering
  \includegraphics[width=0.6\linewidth]{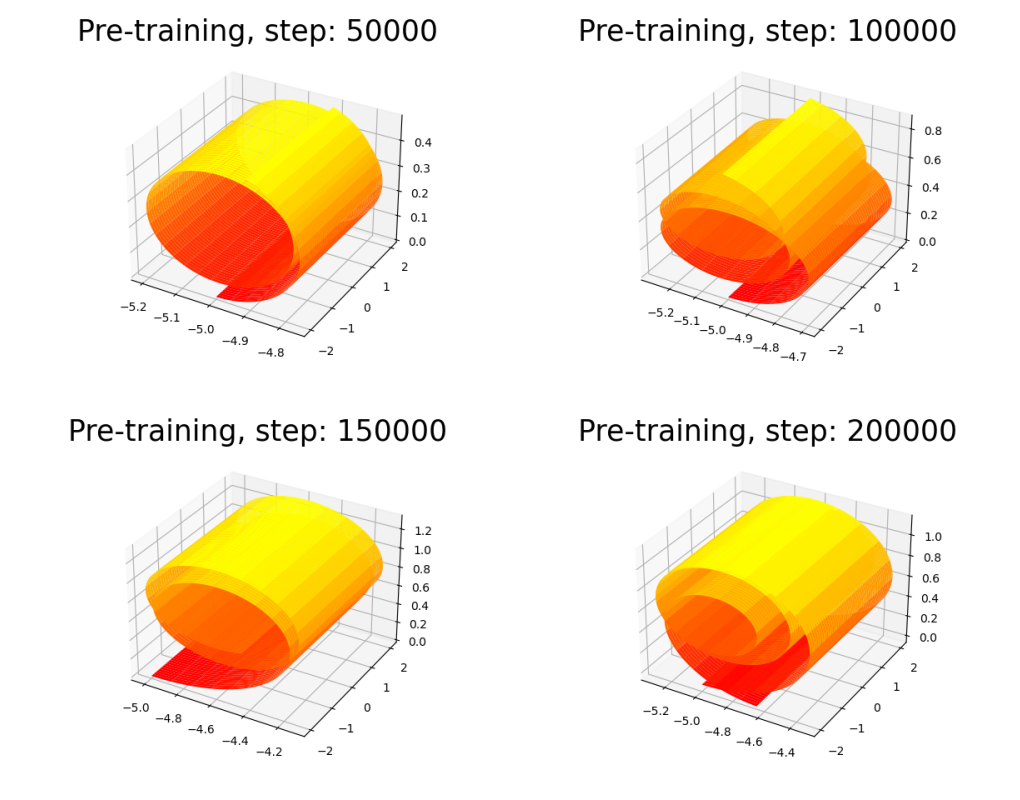}
  \caption{$\alpha = 5$, pre-training provids a good initial guess for the iteration.}\label{fig. pretraina5}
\end{figure}
\begin{figure}[htbp]
  \centering
  \includegraphics[width=0.6\linewidth]{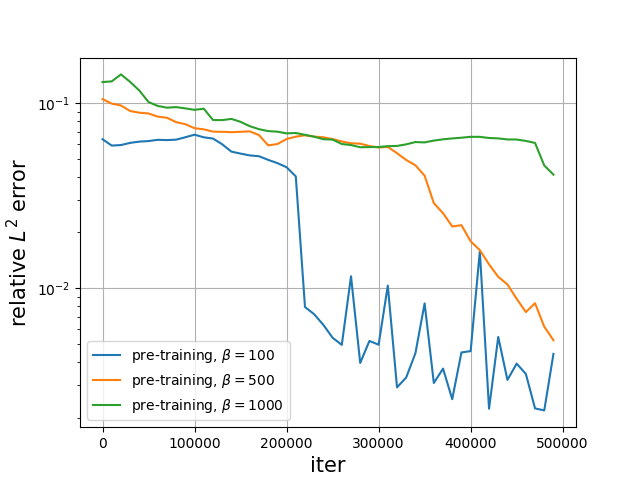}
  \caption{Test error, $\alpha = 5$. After training a total number of $2e5$ steps in the four subdomains in order, we train the network in the whole domain, and report the $L^2$-relative error with the exact solution every $1e4$ steps. For the same value of $\beta$, pre-training greatly speeds up the convergence and improves the accuracy.}
  \label{fig:prea5}
\end{figure}
\begin{table}[htbp]
  \centering
  \begin{tabular}{ccccc}
    \toprule
    $\beta$ & $E[\hat{u}]$ & $C[\hat{u}]$ & $e_{L^2}$ &Cylindrical shape\\
    \midrule
    100    & 500.02& 7.48e-2&4.40e-3&Y \\
    500    & 500.05& 3.17e-2& 5.23e-3&Y \\
    1000  & 500.96& 2.50e-2 & 4.72e-2&Y\\
    \bottomrule
  \end{tabular}
  \caption{Test error, $\alpha = 5$. We train the network $5e5$ steps in the whole domain after $2e5$ steps pre-training, and we report the energy, $L^2$-isometric error, $L^2$-relative error and whether the plate reaches a cylindrical shape.}\label{tab.prefinala5}
\end{table}

We test $\alpha=10$ in the next example, which has not appeared in the previous simulations.
\begin{example}
$Z = -10\textup{Id}_2$.
\end{example}

\textsc{Schmidt}~\cite{schmidt2007minimal} has shown that the global minimizer
of~\eqref{eq:plate} for this problem is given by a cylinder of height $4$ and radius $0.1$ with an energy of $2000$. No previous works have considered such a complicate benchmark example. The bilayer plate needs to roll 16 turns to reach the global minimizer. 

We firstly train the network without the pre-training step and illustrate the results in Figure~\ref{fig. testa10} and Table~\ref{tab.fina10}. It is clear that the method fails to yield the cylindrical shape unless the penalty parameter is as big as $10000$. For small $\beta$, the method seems to be trapped in a local minimizer basin as shown in Figure~\ref{fig:a10dogears}. This is probably due to the fact that the much larger deformation makes it hard to maintain isometric constraints for small $\beta$. To tackle this problem, we exploit the pre-training.

We still divide the domain equally into $5$ blocks; See, Figure~\ref{fig:DDrec}, take $Epoch_{pre}=50000$ and pre-train the 
neural network with a total number of $2e5$ steps. The pseudo-evolutions of the plate with $\beta=1000$ are shown in Figure~\ref{fig:a10b1000predeformation}, the plate bends towards the cylindrical shape. It follows from Figure~\ref{fig:testa10pre} and Table~\ref{tab.testa10Pre} that the bilayer plate reaches the absolute minimizer with an energy $2000.76$ and the relative $L^2$-error is $1.57e-2$ when $\beta=1000$.
\begin{figure}[htbp]
  \centering
  \subfigure[]{
    \includegraphics[width=0.45\linewidth]{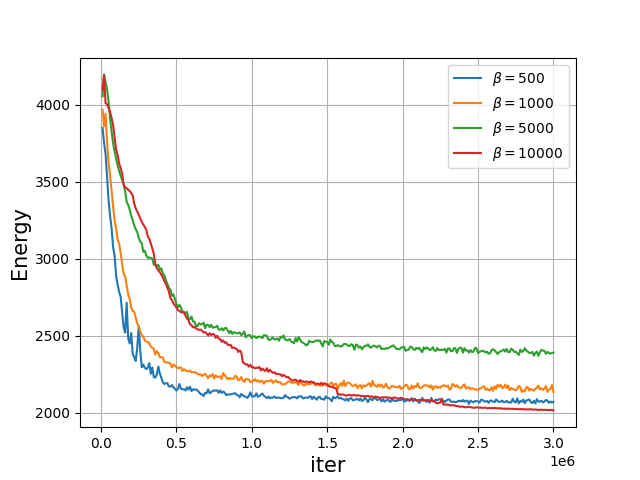}
    \label{Figa10E}
  }
  \subfigure[]{
    \includegraphics[width=0.45\linewidth]{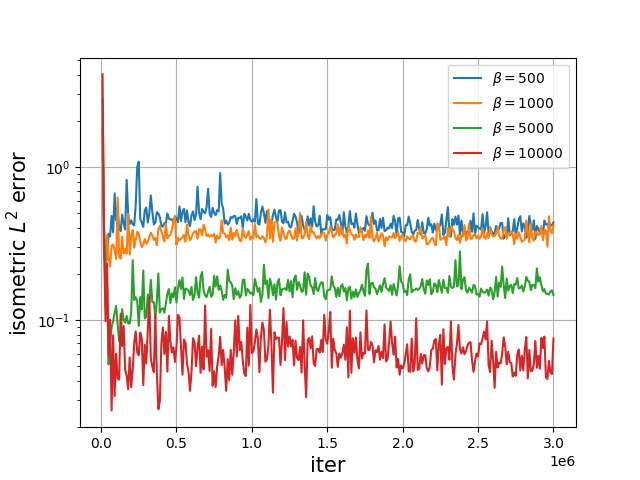}
    \label{Figa10I}
  }
\caption{Training loss, $\alpha= 10$. Decay of energy $E[\hat{u}]$ and $L^2$-isometric error $C[\hat{u}]$ for different penalized parameter $\beta$ during the training process. For $\alpha=10$, if we train the network directly, we need to take $\beta=1e4$ to maintain the isometric constraint.}\label{fig.traina10}
\end{figure} 
\begin{figure}[htbp]
  \centering
  \includegraphics[width=0.6\linewidth]{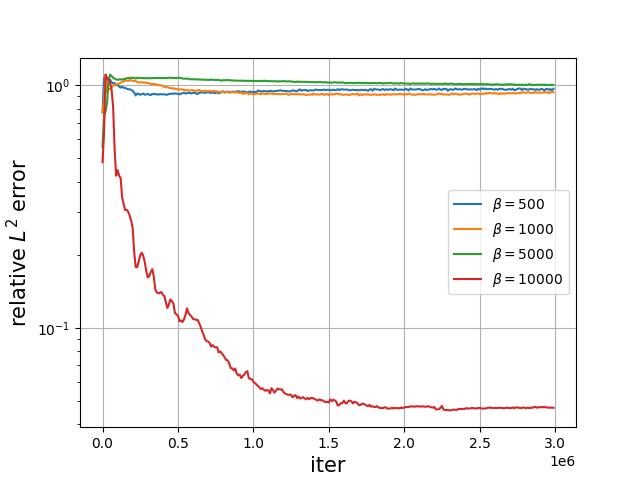}
\caption{Test error, $\alpha= 10$. We report the $L^2$-relative error with the exact solution every 10000 steps. For $\alpha=10$, the bilayer plate also reaches a cylindrical shape when $\beta$ equal to 10000.}\label{fig. testa10}
\end{figure}
\begin{table}[htbp]
  \centering
  \begin{tabular}{ccccc}
    \toprule
    $\beta$ & $E[\hat{u}]$ & $C[\hat{u}]$ & $e_{L^2}$ &Cylindrical shape\\
    \midrule
    500 & 2068.84&4.11e-1 & 9.67e-1 &N\\
    1000 & 2129.63& 3.68e-1& 9.35e-1 &N\\
    5000 & 2371.09& 1.50e-1 & 1.01&N\\
    10000 & 2018.27
    & 5.38e-2& 4.67e-2&Y\\
    \bottomrule
  \end{tabular}
  \caption{Test error, $\alpha = 10$. We report the energy, $L^2$-isometric error, $L^2$-relative error and whether the plate reaches a cylindrical shape after $3e6$ iteration steps. The bilayer plate seems to fall into a local minimizer for small $\beta$, when $\alpha=10$; cf. Figure~\ref{fig:a10dogears}.}\label{tab.fina10}
\end{table}
\begin{figure}[htbp]
  \centering
  \includegraphics[width=0.6\linewidth]{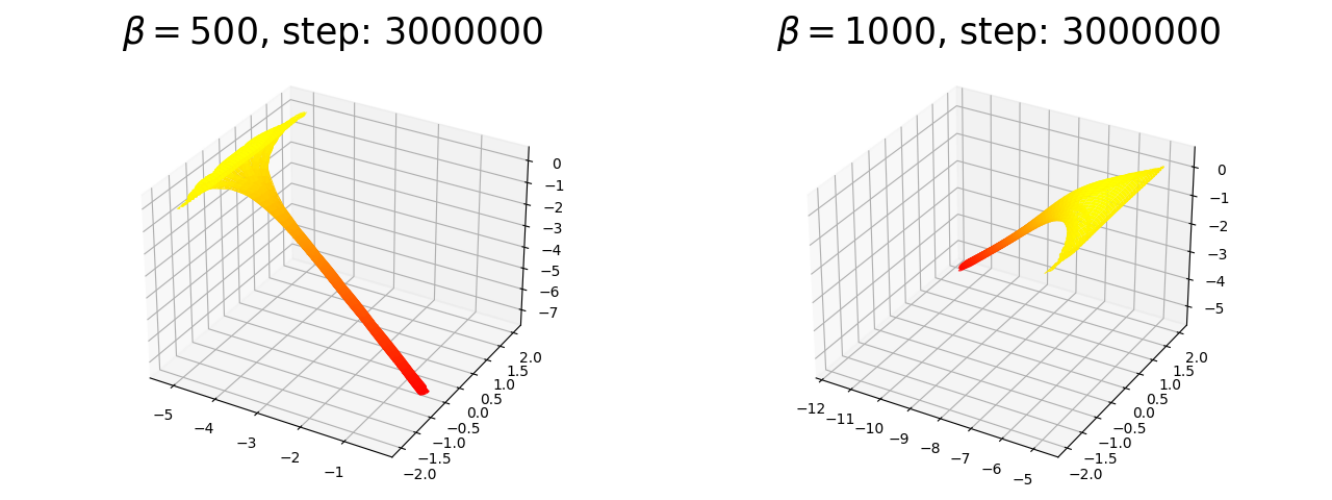}
\caption{$\alpha = 10$. Training the network directly, the 
plate seems to be trapped in a local minimizer.}\label{fig:a10dogears}
\end{figure}
\begin{figure}[htbp]
  \centering
  \includegraphics[width=0.6\linewidth]{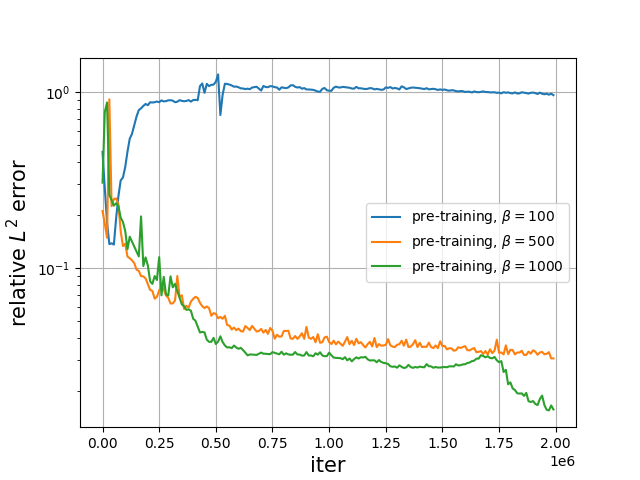}
  \caption{Test error, $\alpha= 10$. After training a total of $2e5$ steps in the four subdomains in order, we train the network in the whole domain, and report the $L^2$-relative error with the exact solution every $1e4$ steps.}
  \label{fig:testa10pre}
\end{figure}
\begin{figure}[htbp]
  \centering
  \includegraphics[width=0.9\linewidth]{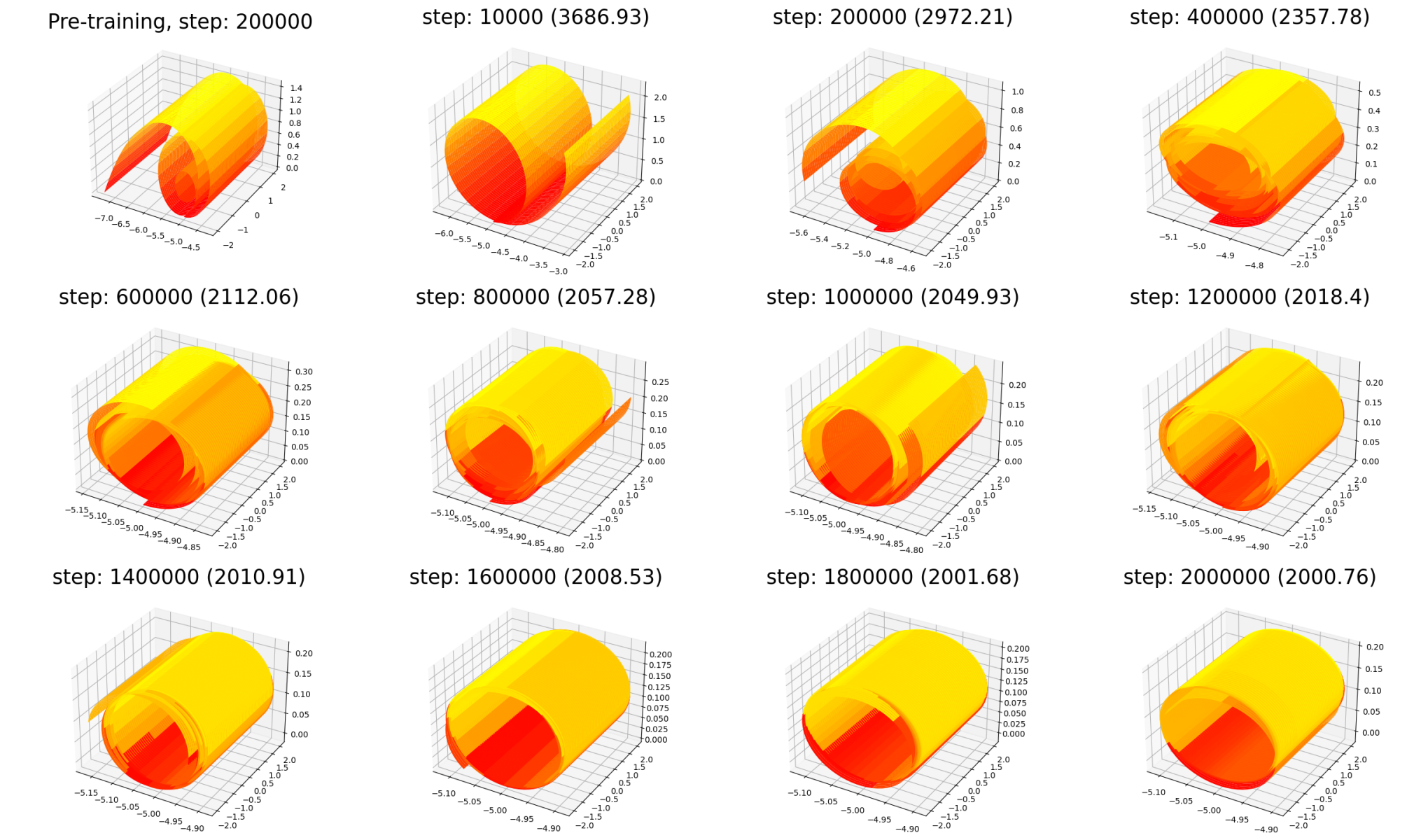}
  \caption{$\alpha = 10$ and $\beta = 1000$. Pre-training finds a good initial deformation, while the plate without pre-training seems to be trapped in a local minimizer; cf. Figure~\ref{fig:a10dogears}.}
  \label{fig:a10b1000predeformation}
\end{figure}
\begin{table}[htbp]
  \centering
  \begin{tabular}{ccccc}
    \toprule
    $\beta$ & $E[\hat{u}]$ & $C[\hat{u}]$ & $e_{L^2}$ &Cylindrical shape\\
    \midrule
    100 & 2119.39&1.18& 9.63e-1&N\\
    500 & 2043.25&1.63e-1& 3.06e-2 &Y\\
    1000 & 2000.76& 2.99e-2& 1.57e-2 &Y\\
    \bottomrule
  \end{tabular}
\caption{Test error, $\alpha = 10$. We first pre-train the network $2e5$ steps, and report the energy, $L^2$-isometric error, $L^2$-relative error and whether the plate reaches a cylindrical shape after $2e6$ iteration steps. Pre-training allows the bilayer plate to converge to the absolute minimizer when $\beta=500,1000$, and also greatly improves the accuracy of the solution.}\label{tab.testa10Pre}
\end{table}     
\subsection{O-shape domain}      
\begin{example}
Let $\Omega$ be $(-5,5) \times   (-2,2) \backslash(-10/3,10/3) \times   (-4/3,4/3) $, clamped on the left side  $\Gamma_D=\{-5\} \times[-2,2]$ , i.e.,
  \begin{small}
    \begin{equation*}
      u=
            \begin{bmatrix}
      x_1\\
            x_2\\
            0
      \end{bmatrix}_{3\times 1}, \quad \nabla u=
   \begin{bmatrix}
    \textup{Id}_{2}\\
    0
\end{bmatrix}_{3\times 2} \quad \text { on } \Gamma_D,
    \end{equation*}
  \end{small}
and we take $Z = -5\,\textup{Id}_2$.
\end{example}             

O-shape example was firstly proposed in~\cite{Bar17bilayer}, in which an equilibrium state like dog's ears has been captured by the gradient flow algorithm. Since the theoretical absolute minimizer is unknown, we aim to explore the minimizer. To skip the exploration process of how large an appropriate penalty parameter $\beta$ needs to be, we still divide the domain into 5 blocks; See, Figure~\ref{fig:inito}, take $Epoch_{pre}=50000$ and pre-train the neural network with a total number of $2e5$ steps.

We illustrate the pseudo-evolution of the plate in Figure~\ref{fig:preoa5b1000}. It is interesting to note that the bilayer plate still rolls into a cylindrical shape. To test the error, we take the cylindrical shape~\eqref{exact sol} as the reference solution with energy $277.78$.  It follows from Figure~\ref{fig:relerrora5od} and Table~\ref{tab.oA5} that the bilayer plate reaches the absolute minimizer with an energy $278.36$ and the relative $L^2$-error is $4.78e-2$ when $\beta = 1000$.
\begin{figure}[htbp]
\centering
\includegraphics[width=0.6\linewidth]{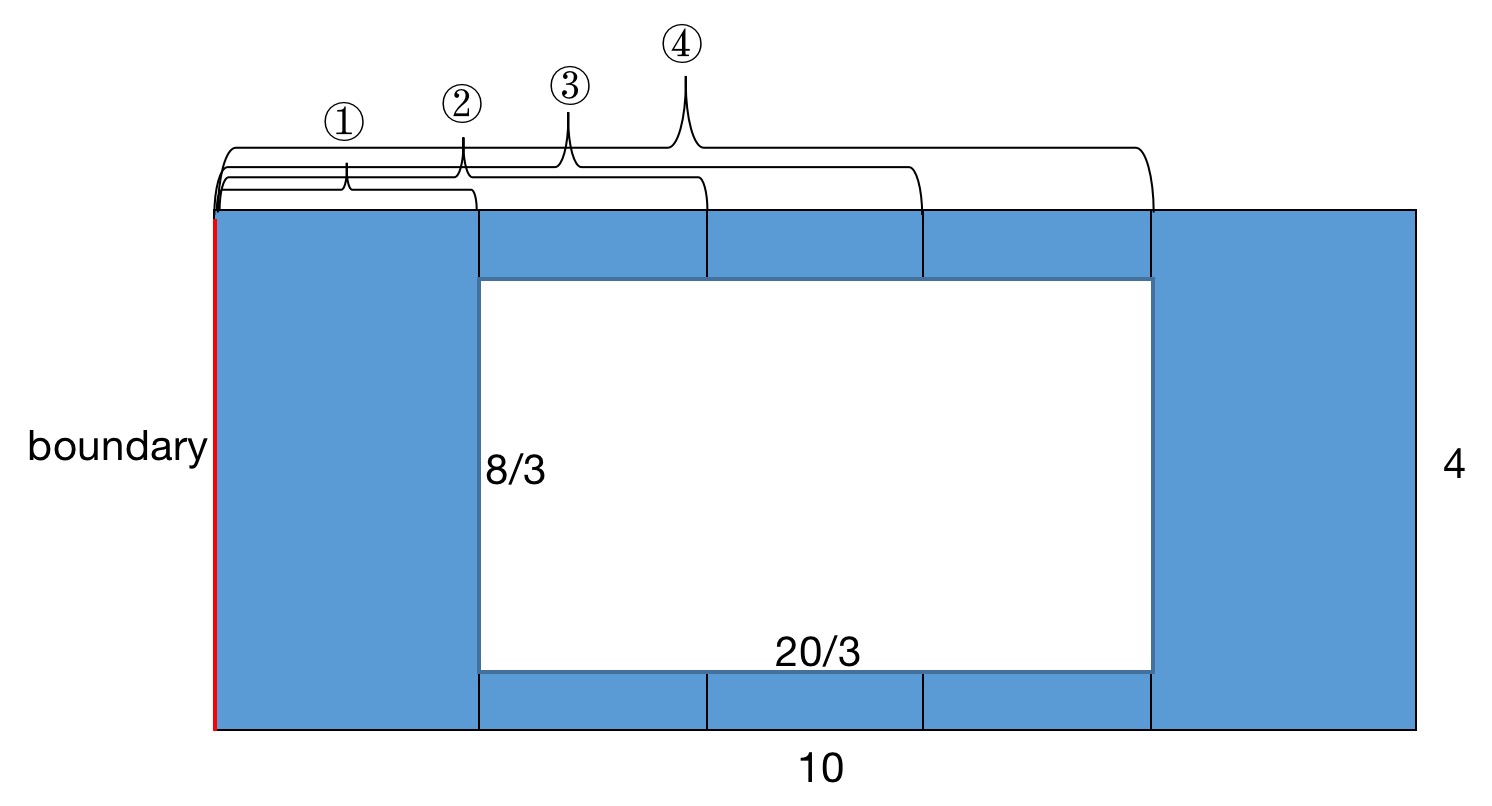}
\caption{To pre-train the network, we divide the domain $(-5,5) \times   (-2,2) \backslash(-10/3,10/3) \times   (-4/3,4/3) $ into 5 parts. We first  randomly sample points in subdomain $1$ and train $Epoch_{pre}$ steps, and then continue the same process in subdomains $2, 3$, and $4$. }
  \label{fig:inito}
\end{figure}
\begin{figure}[htbp]
  \centering
  \includegraphics[width=0.9\linewidth]{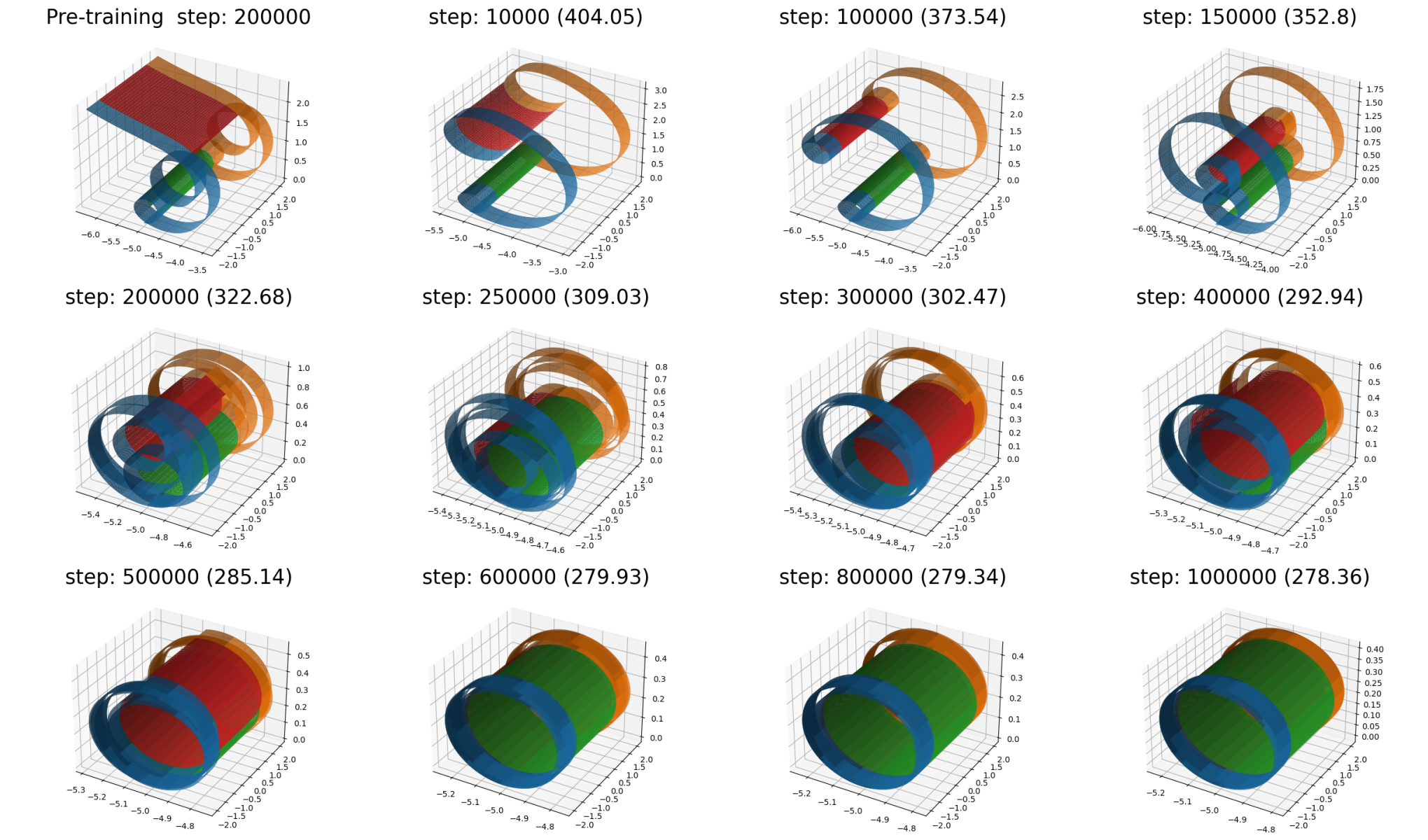}
  \caption{O-shape: $\alpha = 5$ and $\beta = 1000$. To make it easier to distinguish, we use different colors to mark the different parts. The pre-training finds a good initial deformation, and the bilayer plate finally reach a cylindrical shape.}
  \label{fig:preoa5b1000}
\end{figure}
\begin{figure}[htbp]
  \centering
  \includegraphics[width=0.6\linewidth]{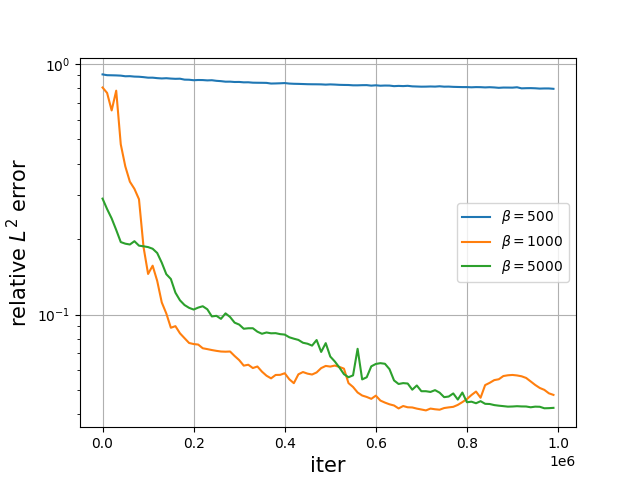}
  \caption{Test error, $\alpha= 5$, O-shape. After training a total of $2e5$ steps in the four subdomains in order, we train the network over the whole domain, and we report the $L^2$-relative error with the exact solution every $1e4$ steps.}
  \label{fig:relerrora5od}
\end{figure}
\begin{table}[htbp]
  \centering
  \begin{tabular}{ccccc}
    \toprule
    $\beta$ & $E[\hat{u}]$ & $C[\hat{u}]$ & $e_{L^2}$ &Cylindrical shape\\
    \midrule
    500 & 269.74&1.13e-1& 7.96e-1 &N\\
    1000 & 278.36& 8.36e-2& 4.78e-2&Y\\
    5000 & 282.34& 2.88e-2&4.24e-2&Y\\
    \bottomrule
  \end{tabular}
\caption{Test error, $\alpha = 5$, O-shape. We first pre-train the network $2e5$ steps, and report the energy, $L^2$-isometric error, $L^2$-relative error and whether the plate reaches a cylindrical shape after $1e6$ iteration steps.}\label{tab.oA5}
\end{table} 
\begin{example}\label{ex:oc}
Let $\Omega$ be $(-5,5) \times   (-2,2) \backslash(-13/3,13/3) \times   (-4/3,4/3) $, clamped on the corner  $\Gamma_D=\{-5\}\times(-2,-4/3)\cup (-5,-13/3)\times\{-2\}$ , i.e.,
  \begin{small}
    \begin{equation*}
      u=
            \begin{bmatrix}
      x_1\\
            x_2\\
            0
      \end{bmatrix}_{3\times 1}, \quad \nabla u=
   \begin{bmatrix}
    \textup{Id}_{2}\\
    0
\end{bmatrix}_{3\times 2} \quad \text { on\quad} \Gamma_D,
    \end{equation*}
  \end{small}
and we take $Z = -\textup{Id}_2$.
\end{example}     
We impose the boundary condition on $\hat{u}(x;\theta)$ by  
\[
\hat{u}(x; \theta) = g_1(x; \hat{\theta}) \hat{y}(x;\theta) + \begin{bmatrix}
      x_1\\
      x_2\\
      0
\end{bmatrix}, 
\] 
where $g_1$ is trained on $\partial\Omega$ by $50000$ steps with least-square loss function approximating
     \[
     	g_1|_{\Gamma_D}=0,\quad\nabla g_1|_{\Gamma_D}=0,\quad g_1|_{\partial\Omega\backslash\Gamma_D}=x_1+x_2+19/3.
     \]
 After $50000$ steps training, the boundary function $g_1$ is smooth and satisfies~\eqref{eq:g1}; See Figure~\ref{fig:g1}. Then we fix $g_1$ and employ the pre-training algorithm. We still divide the domain into $5$ blocks; See Figure~\ref{fig:inito2}, take $Epoch_{pre}=50000$ and pre-train the neural network with a total number of $2e5$ steps.
 
 We illustrate the pseudo-evolution of the plate in Figure~\ref{fig:cornerClamped}. The bilayer plate rolls in two directions, which reaches the final state with an energy $2.29$ and the $L^2$-isometric error is around $6.07e-2$ when $\beta = 500$.
\begin{figure}[htbp]
\centering
\includegraphics[width=0.6\linewidth]{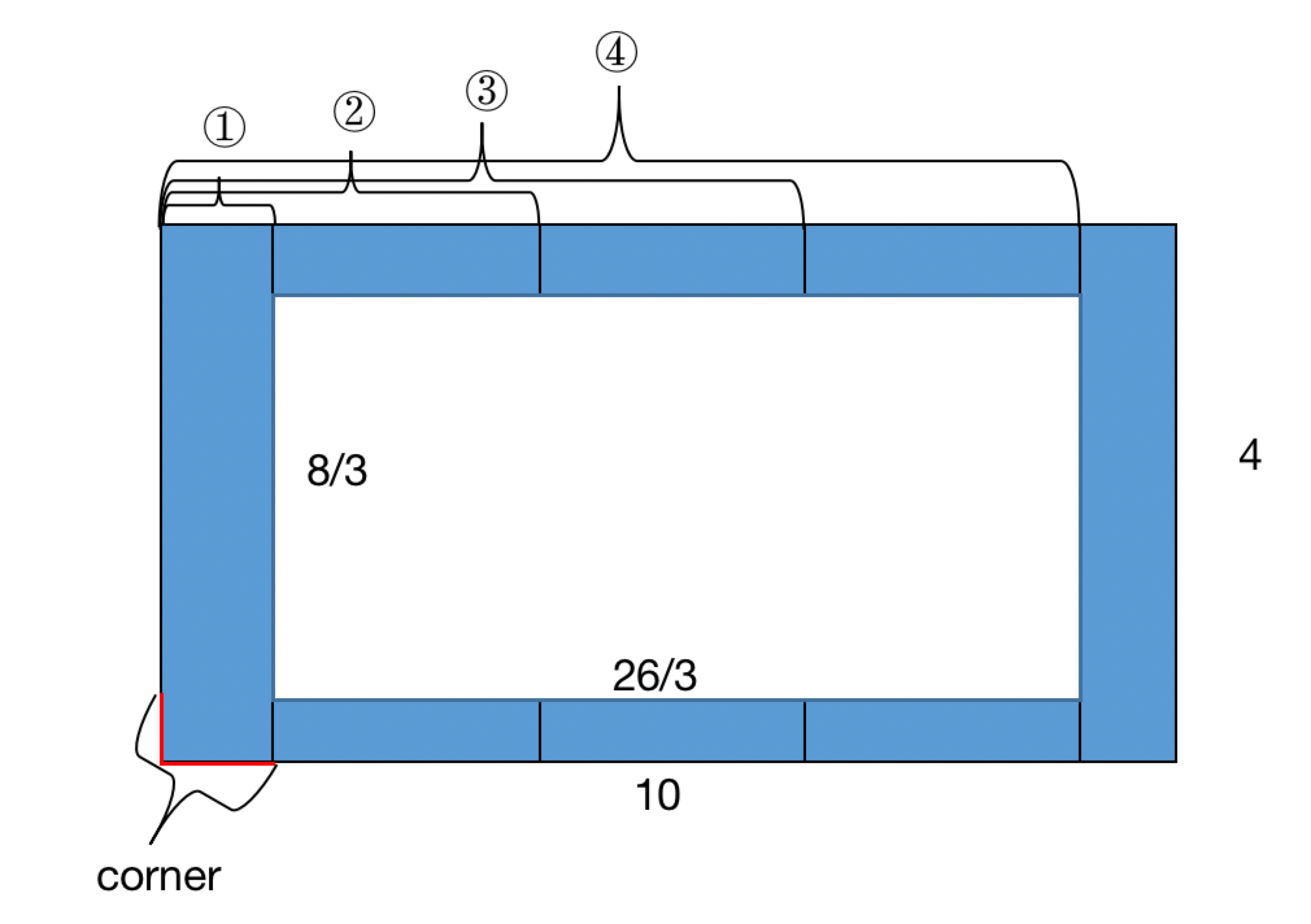}
\caption{To pre-train the network, we divide the domain $(-5,5) \times   (-2,2) \backslash(-13/3,13/3) \times(-4/3,4/3) $ into 5 parts. We randomly sample points in subdomain $1$ and train the network $Epoch_{pre}$ steps, and then apply the same procedure to subdomains $2, 3$, and $4$. }\label{fig:inito2}
\end{figure}
\begin{figure}[htbp]
\centering
\includegraphics[width=0.6\linewidth]{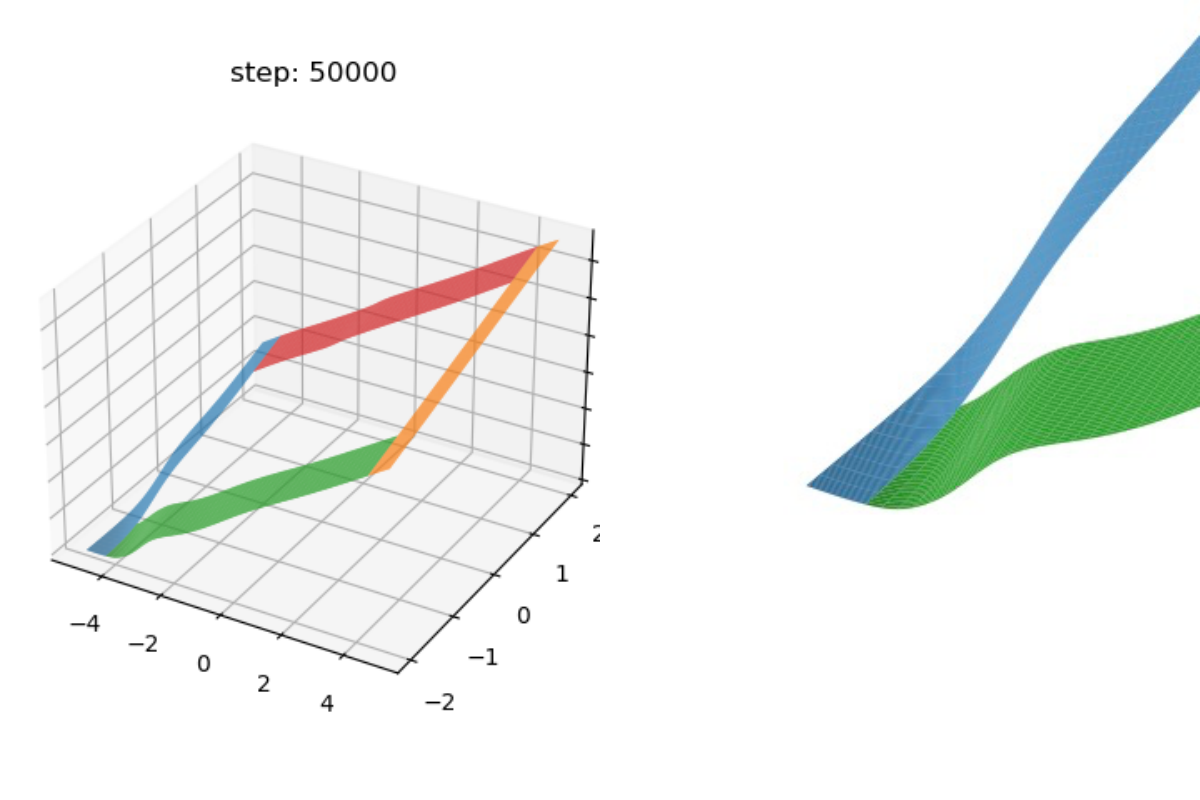}
\caption{$g_1$ is smooth enough and behaves like the distance function to the clamped boundary.}\label{fig:g1}
\end{figure}
\begin{figure}[htbp]\centering
\includegraphics[width=0.9\linewidth]{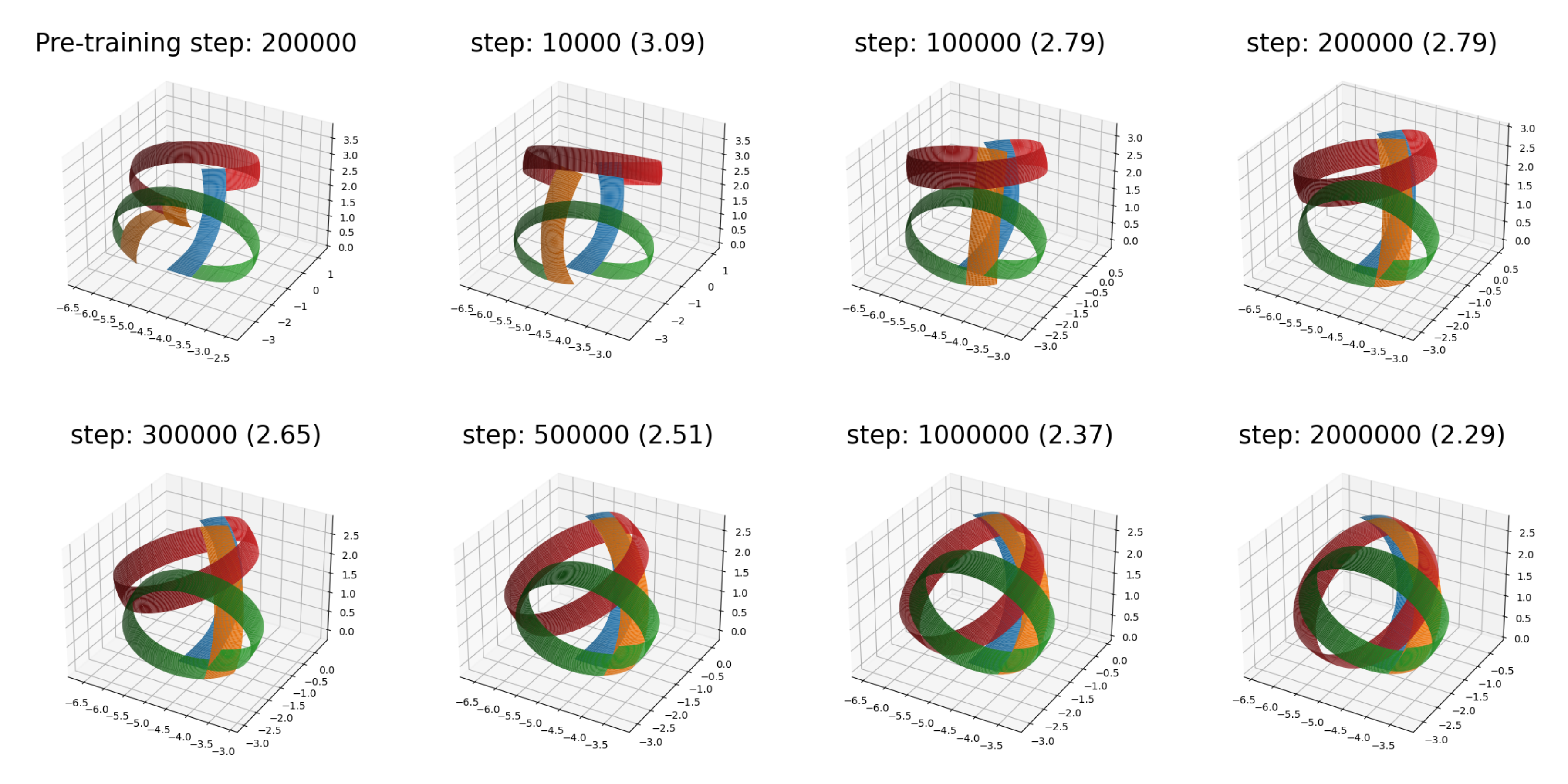}
  \caption{O-shape: $\alpha = 1$ and $\beta = 500$. To make it easier to distinguish, we use different colors to mark the different parts. The pre-training finds a good initial deformation, and the bilayer plate rolls in two directions.}
  \label{fig:cornerClamped}
\end{figure}

Inspired by~\cite{Bar17bilayer} and~\cite{bonito2020discontinuous}, we shall consider the plate with an anisotropic spontaneous curvature.
\subsection{Corkscrew shape}
\begin{example}
The plate $\Omega=(-2. 2) \times   (-3,3)$ is clamped on the side  $\Gamma_D=[-2,2]\times\{-3\}$ , i.e.,
\begin{equation*}
u=\begin{bmatrix}
    x_1\\
    x_2\\
    0
\end{bmatrix}_{3\times 1}, \quad \nabla u=
\begin{bmatrix}
    \textup{Id}_{2}\\
    0
\end{bmatrix}_{3\times 2} \quad \textup{ on } \Gamma_D,
\end{equation*}
and
\begin{equation}
    Z = 
    \begin{bmatrix}
      -3 & 2\\
      2 & -3
    \end{bmatrix},
\end{equation}
which has two eigenvalues $-1$ and $-5$. The angles between the eigenvectors and $x-$coordinate axis are $\pi/4$ and $3\pi/4$, respectively.
\end{example}

We train the neural network directly and illustrate the pseudo-evolution of the plate in Figure~\ref{fig:anicur}. The bilayer plate gradually rolls into a corkscrew shape under anisotropic spontaneous curvature after $2e5$ steps iteration. Since the theoretical absolute minimizer is unknown, we only report the final energy and $L^2$-isometric error. It follows from Table ~\ref{tab.corksrew} that the bilayer plate reaches the corkscrew shape minimizer with an energy $52.53$ and the $L^2$-isometric error is $7.32e-2$ when $\beta = 500$, which is much lower than the best energy $63.31$ reported in~\cite{Bar17bilayer}. 
\begin{figure}[tbph]
  \centering
  \includegraphics[width=0.9\linewidth]{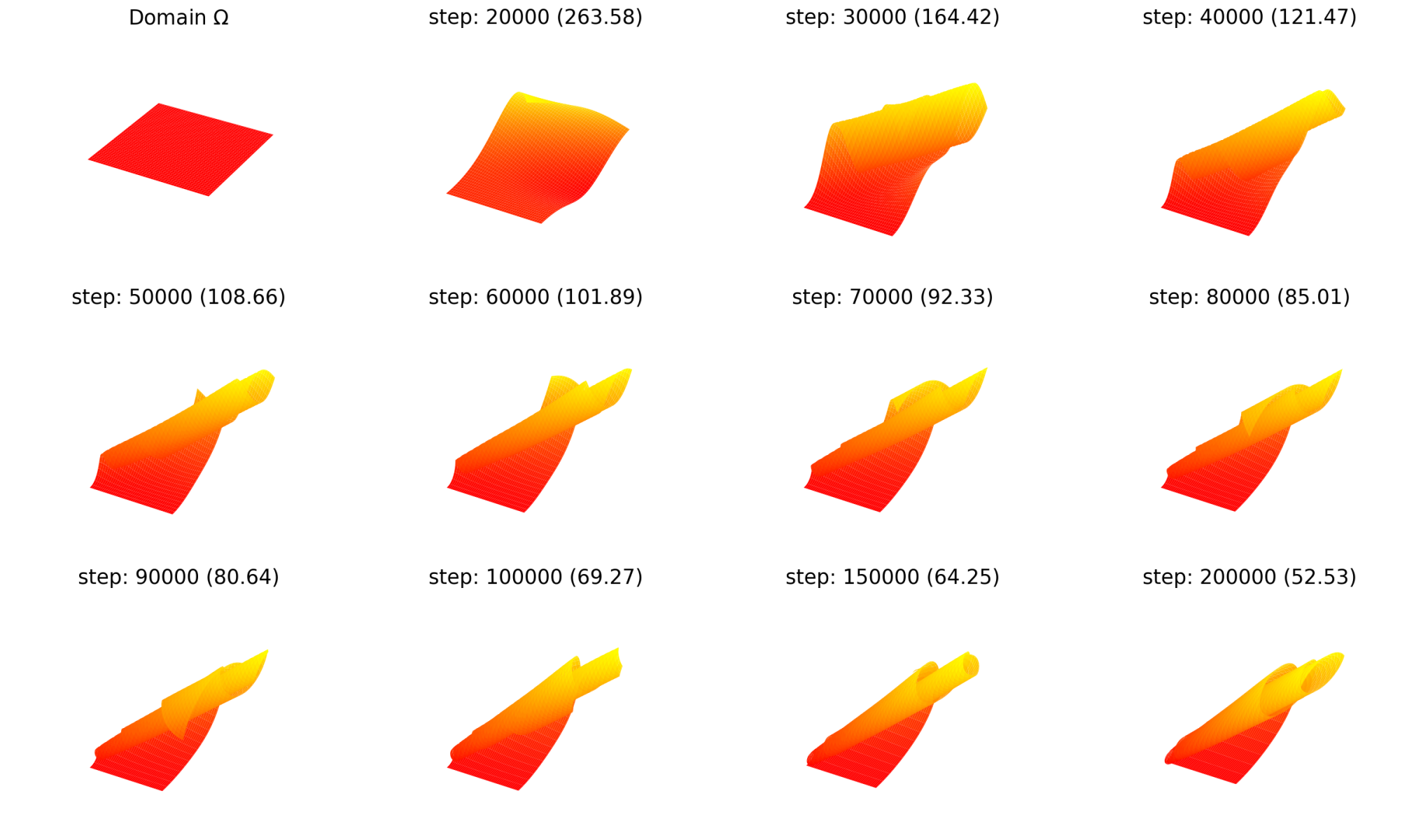}
  \caption{Corkscrew shape, $\beta = 500$. We train the network directly and plot the evolution of the bilayer plate. The number of the iteration steps and the energy are reported.}
  \label{fig:anicur}
\end{figure}
\begin{table}[htbp]
  \centering
  \begin{tabular}{cccc}
    \toprule
    $\beta$ & $E[\hat{u}]$ & $C[\hat{u}]$ & Corkscrew shape\\
    \midrule
    100 & 50.31& 1.21e-1& Y\\
    500 & 52.53& 7.32e-2 &Y\\
    1000 & 55.43& 6.83e-2& Y \\
    \bottomrule
  \end{tabular}
  \caption{Test error, corkscrew shape. We report the energy, $L^2$-isometric error and whether the plate reaches a corkscrew shape after $2e5$ step iterations.}\label{tab.corksrew}
\end{table}

\subsection{Free boundary: anisotropic curvature}
\begin{example}
The plate $(-5,5)\x (-2,2) $ is free of boundary conditions and has the anisotropic spontaneous curvature 
    \begin{equation}
    Z = \begin{bmatrix}
        3&-2\\
        -2&3
    \end{bmatrix},
    \end{equation}
    which has two different eigenvalues $-1$ and $-5$. 
\end{example}

This example is motivated by~\cite{Simpson:2010} in the mechanical engineer literature. We illustrate the pseudo-evolution of the plate in Figure~\ref{fig:cigardeformation}. The plate rolls into the a cigar-like configuration.

We report the energy of the final configuration for different penalty factors $\beta$ in Table~\ref{tab.cigar}. It seems the energy is almost the same with that of the cylindrical shape with radius $1$ in Example~\ref{ex:a1}.
\begin{figure}[tbph]
\centering
\includegraphics[width=0.9\linewidth]{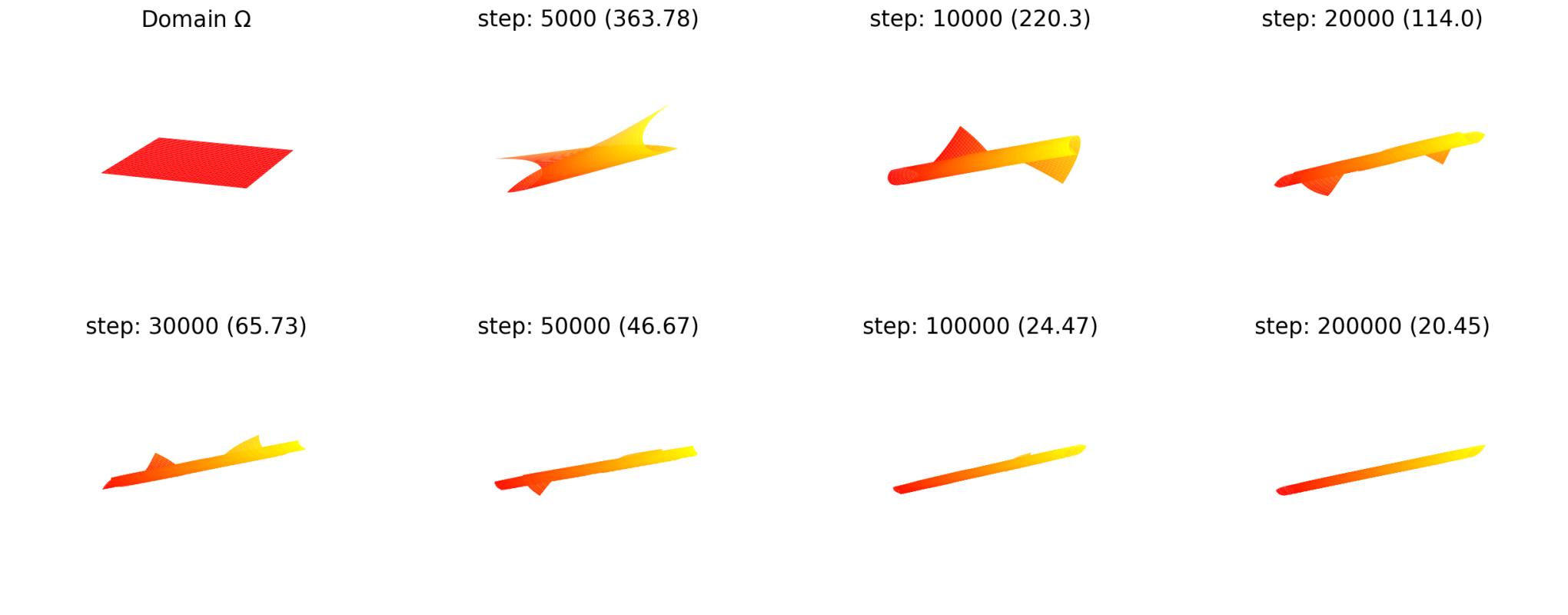}
\caption{Free boundary: a cigar-type configuration, $\beta = 500$. We train the network directly and plot the pseudo-evolution of the bilayer plate. The number of the iteration steps and the energy are reported.}
\label{fig:cigardeformation}
\end{figure}
\begin{table}[htbp]
  \centering
  \begin{tabular}{cccc}
    \toprule
    $\beta$ & $E[\hat{u}]$ & $C[\hat{u}]$ &Cigar-like shape  \\
    \midrule
    100 & 20.04& 5.45e-2 & Y\\
    500 & 20.45& 3.67e-2 & Y\\
    1000 & 20.58&3.31e-2 &Y\\
    \bottomrule
  \end{tabular}
\caption{We report the energy, $L^2$-isometric error and whether the plate reaches a cigar-like configuration after $2e5$ step iterations.}\label{tab.cigar}
\end{table}

As the last example, we report a bilayer plate with free boundary. This example has been observed in~\cite{Jan:2016} using bilayer materials as a building block towards the more complex self-folding structures.
\subsection{Free boundary: helix shape}
\begin{example}
The plate $\Omega=(-8,  8) \times   (-0.5, 0.5)$ has a very high aspect ratio $16$, which deforms with free boundary under the effect of an anisotropic spontaneous curvature 
\begin{equation}
Z = \begin{bmatrix}
      -1 & 3/2\\
      3/2 & -1
    \end{bmatrix}.
\end{equation}
\end{example}

We illustrate the pseudo-evolution of the plate in  Figure~\ref{fig:freeboundary}. The plate rolls into the a DNA-like configuration, which confirms the DNA-like structure reported in~\cite{Jan:2016}. This configuration is also obtained in~\cite{DG2021} by a dG approximation and a gradient flow based minimization algorithm applying to the bilayer plate model~\eqref{eq:plate}.%

We also report the final energy in Figure~\ref{fig:freeboundary}, which is missing in~\cite{Jan:2016} and~\cite{DG2021}. In view of Table~\ref{tab.corksrew}, by contrast to the clamped plate, the isometric constraint may be more easily maintained.
\begin{figure}[tbph]
  \centering
  \includegraphics[width=0.9\linewidth]{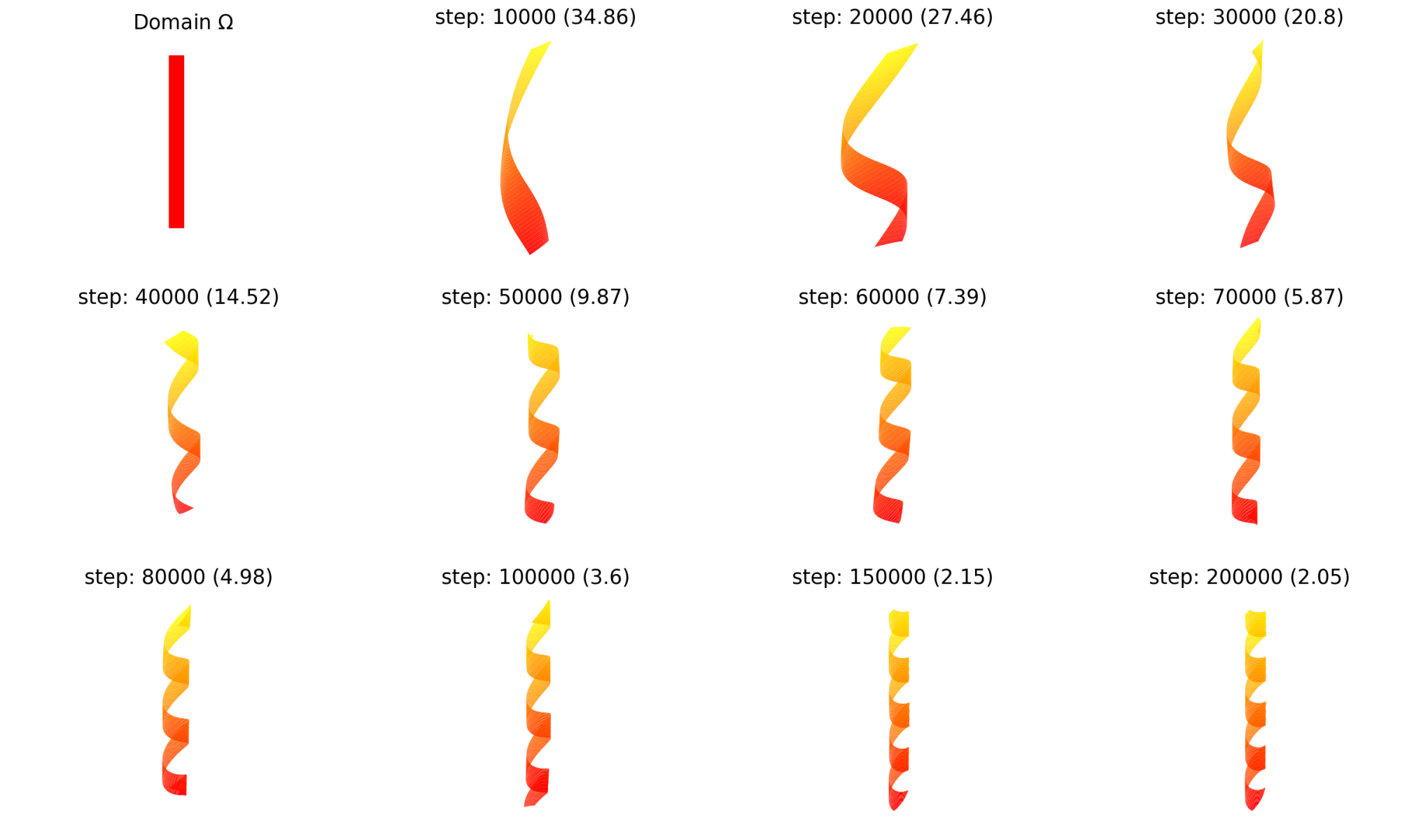}
  \caption{Free boundary: a DNA-like shape, $\beta = 500$. We train the network directly and plot the evolution of the bilayer plate. The number of the iteration steps and the energy are reported.}
  \label{fig:freeboundary}
\end{figure}
\begin{table}[htbp]
  \centering
  \begin{tabular}{cccc}
    \toprule
    $\beta$ & $E[\hat{u}]$ & $C[\hat{u}]$ &DNA-like shape  \\
    \midrule
    100 & 1.87& 1.32e-2 & Y\\
    500 & 1.90& 1.62e-2 & Y\\
    1000 & 1.95&9.12e-3 &Y\\
    \bottomrule
  \end{tabular}
\caption{Test error, DNA-like shape. We report the energy, $L^2$-isometric error and whether the plate reaches a DNA-like shape after $2e5$ step iterations.}\label{tab.dna}
\end{table}
\section{Conclusion}\label{sec:conclusion}
We presents a deep learning method for simulating the large deformation of the bilayer plates. Through extensive numerical experiments, we have demonstrated that our method outperforms the existing techniques in terms of accuracy and convergence. Our approach achieves a significant reduction in the relative error of the energy and effectively maintains the isometric constraint, resulting in an improved accuracy. Moreover, our method exhibits the ability to converge to an absolute minimizer, overcoming the limitations of the gradient flow method. These findings highlight the potential of deep neural networks in accurately simulating complex behaviors of bilayer plates and open avenues for further research in this area.

\bibliographystyle{amsplain}
\bibliography{notes}
\end{document}